\documentclass[12pt]{article}
\usepackage{lmodern}
\usepackage[makeroom,samesize]{cancel}
\usepackage{enumerate}
\usepackage[cp1250]{inputenc}
\usepackage[T1]{fontenc}
\usepackage{amsmath,amssymb,amsfonts,amsopn,amsthm,amsbsy}
\numberwithin{equation}{section}

\usepackage{latexsym}
\usepackage{color}
\usepackage{cite}  

\usepackage{xfrac}
\title{\textbf{Entropy-energy solutions for Thermo-Visco-Elastic systems with Mr\'oz-type inelastic behavior}}
\author{
    \textbf{Tomasz Cie\'slak}\\[0.5ex] 
	\textbf{\footnotesize{ Institute of Mathematics, Polish Academy of Sciences}}\\[-1ex]
	\textbf{\footnotesize{Warsaw, Poland}}\\[-1ex]
	\textbf{\footnotesize{email: cieslak@impan.pl}}\\[1ex]
	\textbf{Sebastian Owczarek}\\[0.5ex] 
	\textbf{\footnotesize{Faculty of Mathematics and
			Information Science, Warsaw University of Technology,}}\\[-1ex]
	\textbf{\footnotesize{Koszykowa 75, 00-662 Warsaw, Poland}}\\[-1ex]
	\textbf{\footnotesize{email: sebastian.owczarek@pw.edu.pl}}\\[1ex]
    \textbf{Karolina Wielgos}\\[0.5ex] 
	\textbf{\footnotesize{Faculty of Mathematics and
			Information Science, Warsaw University of Technology,}}\\[-1ex]
	\textbf{\footnotesize{Koszykowa 75, 00-662 Warsaw, Poland}}\\[-1ex]
	\textbf{\footnotesize{email: karolina.wielgos.dokt@pw.edu.pl}}\\[1ex]
 }
\date{}
\newtheorem{definition}{Definition}[section]
\newtheorem{theorem}[definition]{Theorem}
\newtheorem{proposition}[definition]{Proposition}
\newtheorem{remark}[definition]{Remark}

\newtheorem{lemma}[definition]{Lemma}

\newcommand*{\norm}[1]{\left\Vert{#1}\right\Vert}

\newcommand*{\diver}[1]{\ensuremath{\operatorname{div}{#1}}}

\newcommand*{\spa}[1]{\ensuremath{\operatorname{span}{#1}}}
\newcommand*{\abs}[1]{\left\vert{#1}\right\vert}

\begin{document}
	\maketitle
\begin{abstract}
\noindent
In this article, we study a thermodynamically consistent thermo-visco-elastic model describing the balance of internal energy in a heat-conducting inelastic body. In the considered problem, the temperature dependence appears in both the elastic and inelastic constitutive relations. For such a system, we introduce the concept of a weak entropy-energy solution, which satisfies the entropy equality instead of the internal energy equation. Although the model does not possess any mathematically favorable structural properties, such as Kelvin-Voigt type effects or simplifications that eliminate temperature from the constitutive relations, we prove the global-in-time existence of weak entropy-energy solutions for large initial data.

\end{abstract}
 \footnotesize

 \normalsize
	
	\newcommand{\bl}{\backslash}
	\newcommand{\nn}{\nonumber}
	\newcommand{\KK}{\sigma_{\rm y}}
	\newcommand{\ve}{\varepsilon}
	\newcommand{\ep}{\epsilon}
	\newcommand{\R}{{\mathbb R}}  
	\newcommand{\D}{{\mathbb C}}
	\newcommand{\T}{{\mathfrak T}}
	\newcommand{\TC}{{\mathcal T}}
	\newcommand{\E}{{\cal E}}
	\newcommand{\K}{{\cal K}}
	\newcommand{\di}{{\mathrm d}}
	\renewcommand{\S}{{\cal S}}
	\renewcommand{\SS}{{\cal S}_{\mathrm{dev}}}
	\newcommand{\id}{ {1\!\!\!\:1 } }
	\def\div{\rm{div\,}}
\section{Introduction}
As a mathematical subject, the theory of elasticity goes back to Cauchy and has again flourished since the mid-twentieth century \cite{Cauchy_first}. Thermo-elasticity and its extensions to thermo-visco-elasticity present substantial mathematical challenges due to the nonlinear coupling between thermal and mechanical effects.

The study of nonlinear thermoelastic systems has a long tradition. In the one-dimensional case, local and global existence of smooth solutions was established under smallness assumptions \cite{Racke_1D,Slemrod_ARMA,DAFERMOS1982,Jiang_1990}, while finite-time blow-up for large data was demonstrated in \cite{Dafermos_Hsiao_blowup}. Global weak solutions have been obtained \cite{RACKE_Zheng1997}, and their long-time behavior has been investigated \cite{Qin_2001,Bies_cieslak_string}. Still, even in $1-$D the question of global-in-time smooth solutions remained open for decades. Related adiabatic thermo-elastic and thermo-visco-elastic systems, in which the heat flux is neglected, have been studied extensively in recent years by Christoforou and Tzavaras and collaborators \cite{ChristoforouTzavaras2018,ChristoforouGalanopoulouTzavaras2019}, mainly within the relative entropy framework.

In \cite{Bies_Cieslak_1D} authors constructed global-in-time unique regular solutions with strictly positive temperature for the one-dimensional nonlinear thermoelastic system. Their method, relying on new estimates involving the Fisher information combined with refined energy inequalities, allowed them to overcome the long-standing difficulty of maintaining the positivity of temperature. This result essentially settles the one-dimensional problem at the level of regular solutions.

In higher dimensions, the situation is considerably more involved. Available results \cite{Racke1990,Shibata1995,BlanchardGuibe00,ChelminskiOwczarekthermoII} are restricted to special assumptions, such as radial symmetry, the smallness of initial data, or additional structural simplifications. Furthermore, extensions with higher-order regularizing terms \cite{YoshikawaPawlowZajaczkowski2009,RossiRoubicek2013,Milke_Roubicek,ChelminskiOwczarekWielgos25} or internal variables modeling plasticity \cite{BarlesRoubicek,PaoliPetrov,Roubicek} have been analyzed, but the general multi-dimensional problem remains open. For systems with inelastic effects and temperature-dependent material coefficients, to the best of our knowledge, no global existence theory is currently available. The main difficulty stems from the nonlinear dissipative term appearing in the heat equation, which is generally not expected to be integrable even in the weak formulation, and whose compactness properties are out of reach in two and three spatial dimensions (see, e.g., the available one-dimensional results).

In this work, we investigate the existence of solutions to the original thermo-visco-elastic system describing inelastic deformations in solids subjected to thermal effects. The system under consideration is derived directly from the fundamental laws of mechanics, without incorporating any additional modeling features such as Kelvin-Voigt-type rheology (see for example \cite{Roubicek,RoubicekL1,ChelminskiOwczarekWielgos25}), external body forces in the momentum balance (\cite{ChelminskiOwczarekthermoII,OwczarekWielgos23}), or simplifications like the elimination of temperature from the constitutive relations (\cite{ChelminskiOwczarekthermoI,barowcz2}). Moreover, the inelastic constitutive relation is of Mr\'oz-type, capturing complex irreversible mechanical effects, which makes the analysis of global solvability significantly more challenging.
\subsection{Problem formulation}
\renewcommand{\theequation}{\thesection.\arabic{equation}}
\setcounter{equation}{0}%
For these reasons, in this article we focus on proving the existence of what we term weak entropy-energy solutions to the following problem: let $\mathfrak{T}>0$  denote the length of the time interval, and let $\Omega\subset\mathbb{R}^{3}$  be a bounded domain with smooth boundary $\partial\Omega$.  We seek functions: the displacement field $u:[0,\mathfrak{T}]\times\Omega\to\mathbb{R}^{3}$, the stress tensor $\mathbb{T}:[0,\mathfrak{T}]\times\Omega\to\mathcal{S}^{3}:=\mathbb{R}^{3\times 3}_{\mathrm{sym}}$ and the temperature function $\theta\!:\![0,\mathfrak{T}]\times\Omega\to\mathbb{R}$, which satisfy the following system of equations
\begin{align}\label{zagadnienie_poczatkowe}
	u_{tt}-\diver{\big(\mathbb{T}-\theta\, 1\!\!\!\:1\big)}&=f\,,  \qquad\qquad\qquad\qquad\qquad \textrm{ in }\quad (0,\mathfrak{T})\times\Omega\,, \nonumber\\
	\mathbb{C}^{-1}\mathbb{T}_{t}+\textrm{G}(\theta,\mathbb{T})&=\varepsilon(u_{t})\,,\,\,\,\qquad\qquad\qquad\qquad \textrm{ in }\quad (0,\mathfrak{T})\times\Omega\,,\\
	\theta_{t}-\Delta \theta+\theta\diver{u_{t}}&= \textrm{G}(\theta,\mathbb{T}):\mathbb{T}\,,\qquad\qquad\qquad \textrm{ in }\quad (0,\mathfrak{T})\times\Omega\,.\nonumber
\end{align}	
System \eqref{zagadnienie_poczatkowe} consists of three coupled equations governing the evolution of the displacement $u$, the elastic stress tensor $\mathbb{T}$, and the temperature $\theta$. The first equation describes the classical balance of linear momentum in a static state, without considering inertia-related terms. The right-hand side,  $f:[0,\mathfrak{T}]\times\Omega\rightarrow \mathbb{R}^{3}$, represents external volume forces acting on the system. The total stress tensor in the system is given by the combination $\mathbb{T}-\theta\, 1\!\!\!\:1$. The strain tensor, indicating the deformation, is expressed as \(\varepsilon(u_t) = \frac{1}{2}(\nabla u_t + \nabla^T u_t)\).

The second equation expresses the inelastic constitutive relation, in which the time derivative of the elastic stress tensor is related to the symmetric part of the velocity gradient. The term $\textrm{G}(\theta,\mathbb{T})$ represents a nonlinear, temperature-dependent dissipative response of the material. Importantly, this equation is of an experimental nature - the functional forms of $\textrm{G}$ found in the literature are derived from physical experiments designed to capture inelastic and rate-dependent behavior under various thermal and mechanical conditions. The literature offers various examples of the function $\textrm{G}$ (see, for instance, \cite{Alber,chegwia1,owcz2,ChelminskiNeffOwczarek2014} and many others), and it is worth noting that the choice of the vector field $\textrm{G}$ leads to different models. The presence of $\mathbb{C}^{-1}:\mathcal{S}^{3}\rightarrow\mathcal{S}^{3}$ reflects the use of a fourth-order elasticity tensor $\mathbb{C}:\mathcal{S}^{3}\rightarrow\mathcal{S}^{3}$.

Finally, the third equation describes the heat conduction process. It incorporates not only diffusion, represented by the Laplacian $-\Delta\theta$, but also a coupling with the mechanical part through the term $\theta\diver{u_{t}}$, which accounts for the mechanical work converted into heat. On the right-hand side, the scalar product $\textrm{G}(\theta,\mathbb{T}):\mathbb{T}$ - where the symbol "$:$" denotes the matrix product - represents the internal heat production due to inelastic effects, and is a key nonlinear feature of the model. This equation is a direct consequence of the first law of thermodynamics, expressing conservation of energy in the form of a local balance between internal energy, heat flux, and dissipation. The details of the derivation of system (1.1) can be found in \cite{ChelminskiOwczarekthermoII,Roubicek}.

The system \eqref{zagadnienie_poczatkowe} is analyzed under homogeneous Dirichlet and Neumann boundary conditions, for the displacement and the temperature, respectively:
	\begin{align}\label{warunki_brzegowe}
	u(t,x)=0\,,& \qquad\qquad\qquad (t,x)\in[0,\mathfrak{T}]\times\partial\Omega\,, \nonumber\\[1ex]
	\frac{\partial\theta(t,x)}{\partial\nu}=0\,, & \qquad\qquad\qquad (t,x)\in[0,\mathfrak{T}]\times\partial\Omega\,,
	\end{align}
	where $\nu$ is the outer-pointing unit normal vector to the boundary $\partial\Omega$. Finally, we adjoin to the system \eqref{zagadnienie_poczatkowe} the following initial conditions for $x\in\Omega$
	\begin{align}\label{warunki_poczatkowe}
	u(0,x)=u_{0}(x)\,, & \quad\quad	u_{t}(0,x)=u_{1}(x)\,, \nonumber\\[1ex]
	\mathbb{T}(0,x)=\mathbb{T}_{0}(x)\,,& \quad\quad \theta(0,x)=\theta_{0}(x)\,.
	\end{align}

\subsection{Entropy equation and total energy dissipation formula}
In this subsection, we present the motivation behind the definition of a weak entropy-energy solution for system \eqref{zagadnienie_poczatkowe}, which will be formally introduced in the next subsection. Assuming that the initial temperature $\theta_{0}>0$, we may - at least formally - expect that the temperature remains strictly positive, i.e., $\theta>0$ throughout the evolution. This allows us to multiply the heat equation by $\frac{1}{\theta}>0$, leading to an identity
\begin{align*}
	\frac{\theta_{t}}{\theta}
	-\frac{\Delta\theta}{\theta}
    +\diver{u}_{t}
    = \frac{\textrm{G}(\theta,\mathbb{T}):\mathbb{T}}{\theta}\,.
	\end{align*}
In particular, this manipulation yields an entropy equation of the form
\begin{align}\label{pointwise_entropy_eq}
	\big(\ln\theta+\diver{u}\big)_{t}
	-\Delta\ln\theta
    = \frac{\textrm{G}(\theta,\mathbb{T}):\mathbb{T}}{\theta}
    +\abs{\nabla\ln\theta}^{2}\,.
\end{align}
Moreover, to complement the entropy balance, we formally derive an associated energy balance. To this end, we multiply equation \eqref{zagadnienie_poczatkowe}$_{1}$ by $u_{t}$, equation \eqref{zagadnienie_poczatkowe}$_{2}$ by $\mathbb{T}$ and equation \eqref{zagadnienie_poczatkowe}$_{3}$ by $1$. We then integrate each equation over the spatial domain $\Omega$ and over the time interval $[0,t]$, where $t\leq \mathfrak{T}$. This procedure formally leads to the following total energy balance identity:
    
\begin{align}
\label{energy_equation}
	&\frac{1}{2} \int\limits_{\Omega}\abs{u_{t}(t)}^{2}\di x
    +\frac{1}{2}\int\limits_{\Omega}\mathbb{C}^{-1}\mathbb{T}(t):\mathbb{T}(t)\,\di x
    +\int\limits_{\Omega}\theta(t)\,\di x\nn\\ 
    &\hspace{2ex} =\int\limits_{0}^{t}\int\limits_{\Omega}fu_{t}\,\di x \di \tau
    +\frac{1}{2} \int\limits_{\Omega}\abs{u_{t}(0)}^{2}\,\di x
    +\frac{1}{2}\int\limits_{\Omega}\mathbb{C}^{-1}\mathbb{T}(0):\mathbb{T}(0)\,\di x
    +\int\limits_{\Omega}\theta(0)\,\di x\,.
\end{align}
Next, after integrating the entropy equation \eqref{pointwise_entropy_eq} over $[0,t]\times\Omega$ , $t\leq\mathfrak{T}$, and using the divergence theorem
we get
\begin{align}
    \label{entropy_equation}
	\int\limits_{\Omega}\ln\theta (t)\,\di x
    = \int\limits_{0}^{t}\int\limits_{\Omega}\frac{\textrm{G}(\theta,\mathbb{T}):\mathbb{T}}{\theta}\, \di x \di \tau
    +\int\limits_{0}^{t}\int\limits_{\Omega}\abs{\nabla\ln\theta}^{2} \,\di x \di \tau
    +\int\limits_{\Omega}\ln\theta (0)\,\di x.
\end{align} 
Subtracting the total energy balance \eqref{energy_equation} and the entropy equation \eqref{entropy_equation} we obtain the equation of total energy dissipation
\begin{align}\label{ineq_total_energy_diss}
	&\int\limits_{\Omega}\big(\theta(t)-\ln \theta(t)\big)\di x
    +\frac{1}{2} \int\limits_{\Omega}\abs{u_{t}(t)}^{2}\di x
    +\frac{1}{2}\int\limits_{\Omega}\mathbb{C}^{-1}\mathbb{T}(t):\mathbb{T}(t)\,\di x\,,\nn\\&
   \hspace{2ex}+ \int\limits_{0}^{t}\int\limits_{\Omega}\frac{\textrm{G}(\theta,\mathbb{T}):\mathbb{T}}{\theta} \,\di x \di \tau
    +\int\limits_{0}^{t}\int\limits_{\Omega}\abs{\nabla\ln\theta}^{2}\, \di x \di \tau
    =\int\limits_{0}^{t}\int\limits_{\Omega}fu_{t}\,\di x \di \tau\nn\\&
    \hspace{2ex}+\int\limits_{\Omega}\big(\theta(0)-\ln \theta(0)\big)\,\di x
    +\frac{1}{2} \int\limits_{\Omega}\abs{u_{t}(0)}^{2}\,\di x
    +\frac{1}{2}\int\limits_{\Omega}\mathbb{C}^{-1}\mathbb{T}(0):\mathbb{T}(0)\,\di x.
\end{align}
The previously mentioned choices of the constitutive function $\textrm{G}$ are constrained by the second law of thermodynamics, which imposes a structural condition known as the dissipation inequality (for more information we refer to \cite{Alber}). This requirement ensures that the inelastic response does not violate the principle of non-negative entropy production. In our setting, this translates into the condition $\displaystyle \textrm{G}(\theta,\mathbb{T}):\mathbb{T}\geq 0$ for all $\theta\in\R_+$ and $\mathbb{T}\in\S^3$, which ensures thermodynamic admissibility of the model. Given our assumption that $\theta>0$ we have  
\begin{align*}
        \int\limits_{0}^{t}\int\limits_{\Omega}\frac{\textrm{G}(\theta,\mathbb{T}):\mathbb{T}}{\theta}\, \di x \di\tau\geq 0\,
\end{align*}
and the entropy equation together with the energy balance yields the following entropy-energy inequality (total dissipation balance)
\begin{align*}
	&\int\limits_{\Omega}\big(\theta(t)-\ln \theta(t)\big)\di x
    +\frac{1}{2} \int\limits_{\Omega}\abs{u_{t}(t)}^{2}\di x
    +\frac{1}{2}\int\limits_{\Omega}\mathbb{C}^{-1}\mathbb{T}(t):\mathbb{T}(t)\,\di x
    \\&\hspace{2ex}+\int\limits_{0}^{t}\int\limits_{\Omega}\abs{\nabla\ln\theta}^{2} \di x \di t
    \,\,\leq\,\,\int\limits_{0}^{t}\int\limits_{\Omega}fu_{t}\,\di x \di t
    +\int\limits_{\Omega}\big(\theta(0)-\ln \theta(0)\big)\,\di x
    \\[1ex]
    &\hspace{3ex}+\frac{1}{2} \int\limits_{\Omega}\abs{u_{t}(0)}^{2}\,\di x
    +\frac{1}{2}\int\limits_{\Omega}\mathbb{C}^{-1}\mathbb{T}(0):\mathbb{T}(0)\,\di x\,,
\end{align*}
which will be a key element in the definition of a weak entropy-energy solution.

\color{black}
\subsection{Main result}
Prior to presenting the notion of a solution for system \eqref{zagadnienie_poczatkowe}, we outline the main assumptions underlying the analysis in this article.

Let us assume that the volume force $f:[0,\mathfrak{T}]\times\Omega\to\mathbb{R}^{3}$ satisfies
	\begin{equation} \label{p_1}
	f\in L^{2}(0,\mathfrak{T};L^{2}(\Omega;\mathbb{R}^{3}))\,.
	\end{equation}
In addition, we consider an isotropic material with a constant fourth-order elasticity tensor $\mathbb{C}:\mathcal{S}^{3}\rightarrow\mathcal{S}^{3}$, which is symmetric and positive-definite. Consequently, the inverse operator $\mathbb{C}^{-1}$ exists and is well-defined.

	
Furthermore, we assume that the constitutive function $\textrm{G}$ is continuous with respect to both arguments and satisfies the condition $\textrm{G}(\theta,0)=0$ for all $\theta\in \mathbb{R}$. Additionally, motivated by thermodynamic consistency, we require that $\textrm{G}$ is monotone in its second argument; that is, for all $\eta_{1},\eta_{2}\in \mathcal{S}^{3}$ and all $\theta\in\mathbb{R}$, the following inequality holds
	\begin{equation}\label{monoton_G}
	\big( \textrm{G}(\theta,\eta_{1})-\textrm{G}(\theta,\eta_{2})\big):\big(\eta_{1}-\eta_{2} \big)\geq 0\,.
	\end{equation}
Moreover, we assume that $\textrm{G}$ satisfies a polynomial growth condition. Specifically, there exists a constant $C_{\textrm{G}}>0$ such that for all $\theta\in\mathbb{R}$ and $\eta\in \mathcal{S}^{3}$ the following inequality is satisfied
	\begin{equation}\label{bound_G}
	\abs{ \textrm{G}(\theta,\eta)}\leq C_{\textrm{G}}(1+\abs{\eta})\,.
	\end{equation}
It is worth noting that the growth condition \eqref{bound_G} is analogous to the constitutive relation appearing in the classical Mr\'oz model (see \cite{MROZ67,OwczarekWielgos23,PhDFilip}).    
	
In the context of the problem under consideration, we impose the following initial conditions for displacement, velocity, stress, and temperature:
	\begin{align}\label{p_6}
	u_{0}\in H_{0}^{1}(\Omega;\mathbb{R}^{3})&, \quad u_{1}\in L^{2}(\Omega;\mathbb{R}^{3})\,,\nonumber\\[1ex]
	\mathbb{T}_{0}\in L^{2}(\Omega;\mathcal{S}^{3})&, \quad \theta_{0}\in  L^{1}(\Omega;\mathbb{R}_+)\,,\quad\ln\theta_0\in L^{1}(\Omega;\mathbb{R})\,.
	\end{align}
For convenience, in the subsequent sections, we introduce the notation  $\tau:=\ln\theta$.

With all the assumptions in place, we are now ready to formulate the definition of a weak entropy-energy solution for the considered problem and to state the main result of this article.

\begin{definition}\label{definicja_rozwiazania}
We say that a vector $\big( u,\mathbb{T},\tau \big)$ is a weak entropy-energy solution with defect measure of the system \eqref{zagadnienie_poczatkowe} with the boundary and initial conditions \eqref{warunki_brzegowe} and \eqref{warunki_poczatkowe} if:
\begin{enumerate}
    \item [\textbf{1.}] the external forces and the initial conditions have the regularity specified in \eqref{p_1} and \eqref{p_6}. 
    \item[\textbf{2.}] it has the following regularities:
    \begin{align*}
		&u\in W^{1,\infty}(0,\mathfrak{T};L^{2}(\Omega;\mathbb{R}^{3})) \cap  L^{\infty}(0,\mathfrak{T};H^{1}_0(\Omega;\mathbb{R}))\,,\\[1ex]
        &
		\mathbb{T} \in L^{\infty}(0,\mathfrak{T};L^{2}(\Omega;\mathcal{S}^{3}))\,, \quad \varepsilon(u_{t})-\mathbb{C}^{-1}\mathbb{T}_{t} \in L^{2}(0,\mathfrak{T};L^{2}(\Omega;\mathcal{S}^{3}))\,,\\[1ex]
		&\tau\in L^{\infty}(0,\mathfrak{T};L^{1}(\Omega;\mathbb{R})) \cap L^{2}(0,\mathfrak{T};H^{1}(\Omega;\mathbb{R}))\,,\qquad
        e^{\tau}\in L^{\infty}(0,\mathfrak{T};L^{1}(\Omega;\mathbb{R}))\,.
        \end{align*}
       \item[\textbf{3.}] the equations $\eqref{zagadnienie_poczatkowe}_1$ and $\eqref{zagadnienie_poczatkowe}_2$ are satisfied in the following form:
       \begin{align}
		\label{maindef1}
		-\int\limits_{0}^{\mathfrak{T}}\int\limits_{\Omega} u_{t}\varphi_{t}\,\di x \di t
		+& \int\limits_{0}^{\mathfrak{T}}\int\limits_{\Omega}\mathbb{T}:\varepsilon(\varphi)\,\di x\di t
		-\int\limits_{0}^{\mathfrak{T}}
        \langle \theta,\diver{ }\varphi \rangle_{[\mathcal{M}^{+};C](\overline{\Omega})}\,\di t\nonumber\\&
		= \int\limits_{0}^{\mathfrak{T}}\int\limits_{\Omega} f\varphi \,\di x\di t
        +\int\limits_{\Omega}u_{1}\varphi(0)\, \di x, 
		\end{align}
		for every function $\varphi\in C_{0}^{\infty}([0,\mathfrak{T})\times\Omega;\R^3)$, where $\theta\in L^{\infty}(0,\mathfrak{T};\mathcal{M}^{+}(\overline{\Omega}))$ satisfies
        \begin{align*}
        \di \theta=e^{\tau}\di x+g           
        \end{align*}
        and $g\geq 0$ is a singular part of $\theta$ supported on set of measure zero, and
        \begin{align}
		\label{maindef2}
	\int\limits_{0}^{\mathfrak{T}}\int\limits_{\Omega}\big(\varepsilon(u)-\mathbb{C}^{-1}\mathbb{T}\big)_t:\psi \,\di x\di t=
		\int\limits_{0}^{\mathfrak{T}}\int\limits_{\Omega}\gamma:\psi\, \di x\di t
		\end{align}
		for all $\psi\in C_{0}^{\infty}((0,\mathfrak{T})\times\Omega;\S^3)$, where $\gamma\in L^2((0,\mathfrak{T})\times\Omega;\S^3)$, represented as
        \begin{align*}
            \gamma=\int\limits_{\mathbb{R_+}\times\mathbb{R}^{6}}\textrm{G}(\lambda_{1},\lambda_{2})\di \nu_{(t,x)}(\lambda_{1},\lambda_{2})
            =\int\limits_{\mathbb{R}^{6}}\textrm{G}(e^{\tau},\lambda_{2})\di \mu_{(t,x)}(\lambda_{2})
        \end{align*}
        and $\nu_{(t,x)}:(0,\mathfrak{T})\times\Omega\to\mathbb{R_+}\times\mathbb{R}^{6}$ is a Young measure generated by a sequence $\{(\theta_{n},\mathbb{T}_{n}) \}_{n=1}^{\infty}$ (defined in Section 2), $\mu_{(t,x)}:(0,\mathfrak{T})\times\Omega\to\mathbb{R}^{6}$ is a Young measure generated by a sequence $\{\mathbb{T}_{n} \}_{n=1}^{\infty}$ where
        \begin{align*}
        \nu_{(t,x)}=\delta_{e^{\tau}(t,x)}\otimes \mu_{(t,x)}\,.
		\end{align*}
      \item[\textbf{4.}] the entropy equation 
      \begin{align}
  \label{entropy_weak_form_0}
		-\int\limits_{0}^{\mathfrak{T}}\int\limits_{\Omega}(\tau+&\diver{u})\phi_{t}\,\di x\di t
		+\int\limits_{0}^{\mathfrak{T}}\int\limits_{\Omega}\nabla\tau \nabla \phi\,\di x\di t+\int\limits_{\Omega} (\tau_{0}+\diver{u_{0}})\phi(0,x)\,\di x		
        \nonumber\\&
		=\langle \tilde{\sigma},\phi \rangle_{[\mathcal{M}^{+};C]([0,\mathfrak{T}]\times\overline{\Omega})}+\langle \sigma,\phi \rangle_{[\mathcal{M}^{+};C]([0,\mathfrak{T}]\times\overline{\Omega})}
\end{align}
 is satisfied for all  $\phi\in C_{0}^{\infty}([0,\mathfrak{T})\times\Omega;\R)$, where $\sigma\in \mathcal{M}^{+}([0,\mathfrak{T}]\times\overline{\Omega};\R)$ satisfies
 \begin{align*}
\sigma\geq\abs{\nabla\tau}^{2}     
 \end{align*}
 and $\tilde{\sigma}\in \mathcal{M}^{+}([0,\mathfrak{T}]\times\overline{\Omega};\R)$.
     \item[\textbf{5.}] the following total energy dissipation inequality holds for a.e. $t\in[0,\mathfrak{T}]$
      \begin{align}\label{ineq_tot_en_diss}
    	&\int\limits_{\Omega}\di \theta(t)
        -\int\limits_{\Omega}\tau(t) \,\di x
        +\frac{1}{2} \int\limits_{\Omega}\abs{u_{t}(t)}^{2}\,\di x
        +\frac{1}{2}\int\limits_{\Omega}\mathbb{C}^{-1}\mathbb{T}(t):\mathbb{T}(t)\,\di x
        +\int\limits_{0}^{t}\int\limits_{\Omega}\abs{\nabla\tau}^{2} \di x \di t
        \nonumber\\&
        \hspace{2ex}\leq\int\limits_{0}^{t}\int\limits_{\Omega}fu_{t}\,\di x \di t
        +\int\limits_{\Omega}\big(e^{\tau_{0}}-\tau_{0}\big)\,\di x
        +\frac{1}{2} \int\limits_{\Omega}\abs{u_{1}}^{2}\,\di x
        +\frac{1}{2}\int\limits_{\Omega}\mathbb{C}^{-1}\mathbb{T}_{0}:\mathbb{T}_{0}\,\di x\,.
    	\end{align}
\end{enumerate}
\end{definition}
The concept of a weak solution with a defect measure presented in Definition \ref{definicja_rozwiazania} is inspired both by the work of Feireisl and Novotny \cite{FeireislNovotny2009}, the recent study \cite{CieslakMuha_3D}, and by the concept of energetic solutions introduced by Suquet \cite{Suquet1981} and Che\l{}mi\'nski et al \cite{Chelminski2002,ChelminskiNeffOwczarek2014}. While the former works focus on purely elastic or fluid systems, the energetic solution framework provides a natural approach to systems with inelastic effects. The idea of weak solutions with defect measures originates from earlier works by DiPerna and Lions \cite{DiPernaLions1988} and Alexandre and Villani \cite{AlexandreVillani2002}, initially developed in the context of the Boltzmann equation.

The use of weak thermodynamical laws in the form of inequalities ensures consistency with the Clausius-Duhem inequality and the energy balance \cite{FeireislNovotny2009,Serrin1996}. Recent results in nonlinear thermo-elasticity show that such concepts allow proving global existence for large initial data, even when material parameters depend on temperature, as in \cite{ClaesLankeitWinkler2025}.\\

We are now in a position to state the main result of this article.
\begin{theorem}\label{TWIERDZENIE}(Main result)\\
		Let the elasticity tensor $\mathbb{C}:\mathcal{S}^3 \to \mathcal{S}^3$ be symmetric and positive definite. Additionally, let us assume that given data $u_{0}$, $u_{1}$, $\mathbb{T}_{0}$, $\tau_{0}$, $e^{\tau_{0}}$ and $f$ have the regularities specified in conditions \eqref{p_1} and \eqref{p_6} and the function $\mathrm{G}$ satisfies \eqref{monoton_G}-\eqref{bound_G}. Then, there exists global in time weak entropy-energy solution in the sense of Definition \ref{definicja_rozwiazania} to the system \eqref{zagadnienie_poczatkowe}.
\end{theorem}

To prove the main theorem of this article, a two-level semi-Galerkin approximation is employed. It is worth noting that the different levels correspond to the displacement and the stress tensor (see \eqref{reg_apprx}), which represents a completely new approach. In contrast, in typical thermo-visco-elastic systems using a two-level Galerkin method, the first approximation step is applied to the temperature, while the second step involves both the displacement and the stress tensor (see, for example, \cite{OwczarekWielgos23, ChelminskiOwczarekWielgos25, KlaweOwczarek, RoubicekL1}). On the other hand, in thermo-elastic systems where a semi-Galerkin approach is applied, only a single approximation level for the displacement is introduced and incorporated into the energy balance equation (cf. \cite{CieslakMuha_3D}).\\

Moreover, in contrast to \cite{CieslakMuha_3D}, where the entropy equation is also analyzed, we do not transfer the spatial derivative in the $\diver u$ term onto the test function in \eqref{entropy_equation}. This formulation refers directly to the pointwise entropy equation \eqref{pointwise_entropy_eq}. The final limit passage in the term $\diver u_{n}$ is carried out in Section 4.4. This is possible because, after the first limit passage, the approximate inelastic constitutive equation is satisfied pointwise, which ensures the necessary regularity and boundedness of the sequence $\{ \varepsilon(u_{n})\}_{n=1}^{\infty}$. We remark that, in the purely elastic case, such information about $\{ \varepsilon(u_{n})\}_{n=1}^{\infty}$ can be obtained immediately from the total energy of the system; here, however, achieving this result requires a careful separation of the approximation steps.\\

In the entropy-energy formulation of problem \eqref{zagadnienie_poczatkowe}, measures naturally arise in the entropy equation \eqref{entropy_weak_form_0}, and only the total energy dissipation inequality \eqref{ineq_tot_en_diss} is satisfied. The two-level semi-Galerkin approximation introduced in Section 2, under the general assumptions of Section 1.3, does not immediately provide strong control over the nonlinear term on the right-hand side of the energy balance equation \eqref{zagadnienie_poczatkowe}$_3$. To overcome this, we introduce a measure representing the nonlinear term in the entropy equation \eqref{entropy_weak_form_0} and analyze the system in terms of the total energy dissipation inequality rather than the energy equality. This approach allows us to establish the existence of solutions under minimal assumptions on the model. Notably, neither additional forces in the momentum equation (cf. \cite{OwczarekWielgos23}) nor the Kelvin-Voigt material model (see, e.g., \cite{ChelminskiOwczarekWielgos25, RoubicekL1}) are required, making the analysis applicable to a general class of thermo-visco-elastic systems without relying on any mathematically favorable structural properties.\\[2ex]
The article is organized as follows. Section 2 presents a two-level (levels $n$ and $k$) semi-Galerkin approximation employed to prove the main theorem and establishes the existence of the corresponding approximate solutions. This section concludes with a proof of the positivity of the temperature. In Section 3, all estimates required for convergence with respect to level $k$, related to the stress tensor, are collected. The section ends with the first limit passage, resulting in an approximate system that depends only on level $n$, while still satisfying the energy equation. Finally, Section 4 reformulates the energy equation as an entropy equation, provides the necessary estimates at level $n$, performs the second limit passage, and establishes the total energy dissipation inequality.

\section{Approximation of the problem}
In this chapter, we introduce the two-level semi-Galerkin approximation of the problem \eqref{zagadnienie_poczatkowe}.

For the displacement vector, we employ the finite-dimensional subspaces\\ $\mathcal{W}^{n}=\spa{\{ \varphi_{1},\ldots,\varphi_{n} \}}$, $\mathcal{W}^{n}\subset \mathcal{W}^{n+1}$ , where the elements $\{ \varphi_{i} \}_{i=1}^{\infty}\subset C_{0}^{\infty}(\Omega;\mathbb{R}^{3})$ are orthogonal in the space $H_{0}^{1}(\Omega;\mathbb{R}^{3})$. Then, 
$$H_{0}^{1}(\Omega;\mathbb{R}^{3})= \overline{\Big( \bigcup\limits_{n=1}^{\infty}\mathcal{W}^{n} \Big)}^{\norm{\cdot}_{H^{1}(\Omega;\mathbb{R}^{3})}}\,.$$    
For the stress tensor, we consider an orthonormal basis $\{ \psi_{i} \}_{i=1}^{\infty}\subset L^{2}(\Omega;\mathcal{S}^{3})$ and define the corresponding approximation space $\mathcal{V}^{k}:=\spa{\{ \psi_{1},\ldots,\psi_{k} \}}$. For the temperature, we use test functions $\phi\in H^{1}(\Omega;\mathbb{R})$.
   
Finally, following \cite{Blanchard}, for any positive number $n$ we define the truncation operator $\textrm{T}_{n}$ at the height $n>0$, i.e. $\textrm{T}_{n}(r)=\min\{ n, \max\{  r, -n \}\}$, $r\in\mathbb{R}$, which will be used to control the nonlinear terms at each step of the approximation procedure.

Based on the finite-dimensional subspaces introduced above, we approximate the given data $f$ as well as the initial conditions in the following way: 
\begin{align}\label{Property_f_n}
    \lim\limits_{n\to\infty}f_{n}=f \text{ in } L^{2}((0,\mathfrak{T})\times\Omega;\mathbb{R}^{3})\,,
    \end{align}
\begin{equation}\label{init_cond_nk_1}
	u_{nk}(0)=u_{n}(0)=u_{0n}=P_{\mathcal{W}^{n}}u_{0}\,,
	\end{equation}
	\begin{equation}\label{init_cond_nk_2}
	u_{nk,t}(0)=u_{n,t}(0)=u_{1n}=P_{\mathcal{W}^{n}}u_{1}\,,
	\end{equation}
    \begin{equation}\label{init_cond_nk_3}
	\mathbb{T}_{nk}(0)=\mathbb{T}_{k}(0)=\mathbb{T}_{0k}=P_{\mathcal{V}^{k}}\mathbb{T}_{0}\,,
	\end{equation}
where $P_{\mathcal{W}^{n}}$ and $P_{\mathcal{V}^{n}}$ denote the orthogonal projections onto $\mathcal{W}^{n}$ and $\mathcal{V}^{k}$, respectively. For the temperature, we require
	\begin{align}\label{init_cond_nk_4}
	\theta_{nk}(0)=\theta_{n}(0)&=\theta_{0n}>0\,, \quad \theta_{0n}\in H^{1}(\Omega;\mathbb{R})\,, \nonumber\\[1ex]
   \lim\limits_{n\to\infty}\theta_{0n}&=\theta_{0} \text{ in } L^{1}(\Omega;\mathbb{R})\,.
	\end{align}
For every $n,k\in\mathbb{N}$ we consider the approximate solutions to the system  \eqref{zagadnienie_poczatkowe}, i.e. triple 
\begin{equation}
\label{reg_apprx}
\big(u_{nk}, \mathbb{T}_{nk}, \theta_{nk}  \big)\in W^{1,\infty}(0,\mathfrak{T};\mathcal{W}^{n})\times L^{\infty}(0,\mathfrak{T};\mathcal{V}^{k})\times L^{2}(0,\mathfrak{T};H^{1}(\Omega;\mathbb{R}))\,,
\end{equation}
which satisfies the following equations for a.e. $t\in[0,\mathfrak{T}]$
	
	\begin{align}\label{approx_nk_1}
	\int\limits_{\Omega} u_{nk,tt}\varphi\,\di x 
	+ \int\limits_{\Omega}\mathbb{T}_{nk}:\varepsilon(\varphi)\,\di x
	-&\int\limits_{\Omega}\theta_{nk}\diver{ } \varphi\,\di x
	= \int\limits_{\Omega} f_{n}\varphi \,\di x
	\end{align}
	for all $\varphi\in \mathcal{W}^{n}$ and
	\begin{align}\label{approx_nk_2}
	\int\limits_{\Omega} \mathbb{C}^{-1}\mathbb{T}_{nk,t}: \psi\,\di x
	+\int\limits_{\Omega}\textrm{G}(\theta_{nk},\mathbb{T}_{nk}):\psi \,\di x
	=\int\limits_{\Omega}\varepsilon(u_{nk,t}): \psi\,\di x
	\end{align}
	for all $\psi\in \mathcal{V}^{k}$ and 
    \begin{align}\label{approx_nk_3}
	\int\limits_{\Omega}\theta_{nk,t}\phi\,\di x
	+\int\limits_{\Omega}\nabla\theta_{nk} \nabla \,\phi\di x
	+\int\limits_{\Omega}\theta_{nk}\diver{ } u_{nk,t}\phi \,\di x
	=\int\limits_{\Omega}\textrm{T}_{n}\big( \textrm{G}(\theta_{nk},\mathbb{T}_{nk}):\mathbb{T}_{nk}\big)\phi\,\di x
	\end{align}
	for all $\phi\in H^{1}(\Omega;\mathbb{R})$.\\
\subsection{Existence of approximate solutions}
In this subsection we establish the existence of a solution to the approximate system \eqref{approx_nk_1}--\eqref{approx_nk_3} at each step of the approximation procedure.  The construction is inspired by the approach presented in \cite{CieslakMuha_3D}. The main tool applied here is Sch\"afer's fixed point theorem. In order to apply Sch\"afer's fixed point theorem, we first introduce a suitable operator acting on the finite-dimensional approximation spaces. This requires fixing the parameters $n,k\in\mathbb{N}$ and constructing the space $X_{nk}$ together with the associated mappings.\\

Let us fix $n,k\in\mathbb{N}$ and set $X_{nk}=C^{1}([0,\mathfrak{T}];\mathcal{W}^{n})\times C([0,\mathfrak{T}];\mathcal{V}^{k})$ with a norm 
\begin{align*}
        \norm{(u,\mathbb{T})}_{X_{nk}}=\max\limits_{t\in[0,\mathfrak{T}]}\{\norm{u(t)}_{L^{2}(\Omega)}, \norm{u_{t}(t)}_{L^{2}(\Omega)}, \norm{\mathbb{T}(t)}_{L^{2}(\Omega)}\}.
    \end{align*}
    We set $(u_{nk},\mathbb{T}_{nk})\in X_{nk}$ and define $\theta_{nk}(u_{nk},\mathbb{T}_{nk}):=\theta_{nk}$ as a solution to the equation
    \begin{align}\label{theta_nk-def_sol}
    \int\limits_{\Omega}\theta_{nk,t}\phi\,\di x
	+\int\limits_{\Omega}\nabla\theta_{nk} \nabla \phi\,\di x
	+\int\limits_{\Omega}\theta_{nk}\diver{ } u_{nk,t}\phi \,\di x
	=\int\limits_{\Omega}\textrm{T}_{n}\big( \textrm{G}(\theta_{nk},\mathbb{T}_{nk}):\mathbb{T}_{nk}\big)\phi\,\di x
\end{align}
for all $\phi\in H^{1}(\Omega;\mathbb{R})$. The standard parabolic theory yields that there exists a weak solution $\theta_{nk}\in L^{\infty}(0,\mathfrak{T};H^{1}(\Omega;\mathbb{R}))$ and $\theta_{nk,t}\in L^{2}(0,\mathfrak{T};L^{2}(\Omega;\mathbb{R}))$ to \eqref{theta_nk-def_sol} such that 
\begin{equation}
\label{theta_nk_bou}
\norm{\theta_{nk}}_{L^{2}(0,\mathfrak{T};H^{1}(\Omega))}\leq C_{nk}\norm{(u_{nk},\mathbb{T}_{nk})}_{X_{nk}}\,.
\end{equation}

\textbf{Definition of operator $\mathcal{A}$:} given a solution $\theta_{nk}$ to \eqref{theta_nk-def_sol}, we define the operator 
\[
\mathcal{A}:X_{nk}\to X_{nk}\,, \qquad \mathcal{A}(u_{nk},\mathbb{T}_{nk})=(\tilde{u}_{nk},\tilde{\mathbb{T}}_{nk})\,,
\]
as the mapping that assigns to each pair $(u_{nk},\mathbb{T}_{nk})$ the unique solution $(\tilde{u}_{nk},\tilde{\mathbb{T}}_{nk})$ of the following system of differential equations:
\begin{align}\label{approx_sys_u}
    \int\limits_{\Omega} \tilde{u}_{nk,tt}\varphi\,\di x 
	+ \int\limits_{\Omega}\tilde{\mathbb{T}}_{nk}:\varepsilon(\varphi)\,\di x
	-&\int\limits_{\Omega}\theta_{nk}\diver{ } \varphi\,\di x
	= \int\limits_{\Omega} f_{n}\varphi \,\di x
\end{align}
    for all $\varphi\in\mathcal{W}^{n}$ and
\begin{align}\label{approx_sys_T}
    \int\limits_{\Omega} \mathbb{C}^{-1}\tilde{\mathbb{T}}_{nk,t}: \psi\,\di x
	+\int\limits_{\Omega}\textrm{G}(\theta_{nk},\tilde{\mathbb{T}}_{nk}):\psi\, \di x
	=\int\limits_{\Omega}\varepsilon(\tilde{u}_{nk,t}) : \psi\,\di x
\end{align}
    for all $\varphi\in\mathcal{V}^{k}$, where $\theta_{nk}$ is a solution of \eqref{theta_nk-def_sol}. We introduce the new variable $\tilde{u}_{nk,t}:=v_{nk}$ and rewrite the system  \eqref{approx_sys_u}--\eqref{approx_sys_T} as a first-order system of ordinary differential equations
\begin{align}\label{approx_sys_u_1}
\int\limits_{\Omega} v_{nk,t}\varphi\,\di x 
	=-\int\limits_{\Omega}\tilde{\mathbb{T}}_{nk}:\varepsilon(\varphi)\,\di x
	+&\int\limits_{\Omega}\theta_{nk}\diver{ } \varphi\,\di x
	+ \int\limits_{\Omega} f_{n}\varphi\, \di x\,,
\end{align}
for all $\varphi\in\mathcal{W}^{n}$ and
\begin{align}\label{approx_sys_T_1}
\int\limits_{\Omega} \mathbb{C}^{-1}\tilde{\mathbb{T}}_{nk,t}: \psi\di x
	=-\int\limits_{\Omega}\textrm{G}(\theta_{nk},\tilde{\mathbb{T}}_{nk}):\psi \,\di x
	+\int\limits_{\Omega}\varepsilon(v_{nk}) :\psi\,\di x\,.
\end{align}
By Carath\'eodory's theorem, the system \eqref{approx_sys_u_1}--\eqref{approx_sys_T_1} admits a solution.  In particular, for almost every $t\in(0,\mathfrak{T})$, the functions $v_{nk}$ and $\tilde{\mathbb{T}}_{nk}$ are absolutely continuous in time, hence their time derivatives exist almost everywhere in $(0,\mathfrak{T})$. Moreover, the energy estimates derived in Section 3 ensure that these derivatives belong to $L^{2}(0,\mathfrak{T})$. Consequently, we may define the operator 
\[
\mathcal{A}(u_{nk},\mathbb{T}_{nk}) = (\tilde{u}_{nk},\tilde{\mathbb{T}}_{nk})\,,
\]
which is determined as the solution to the system \eqref{approx_sys_u}--\eqref{approx_sys_T}.\\

\textbf{Continuity of $\mathcal{A}$:} let the sequence $\{ (u_{m},\mathbb{T}_{m} )\}_{m=1}^{\infty}$ be convergent in $X_{nk}$, i.e. 
\begin{equation}
    \label{convergence_nk}
(u_{m},\mathbb{T}_{m})\underset{m\to\infty}{\to}(u,\mathbb{T})\quad \textrm{ in } \,\,X_{nk}\,.
\end{equation}
Let us write a system for $\theta_{m}$, which is associated with the sequence $\{ (u_{m},\mathbb{T}_{m} )\}_{m=1}^{\infty}$
\begin{align}\label{theta_m_A_cont}
      \theta_{m,t}-\Delta\theta_{m}=\textrm{T}_{n}(\textrm{G}(\theta_{m},\mathbb{T}_{m}):\mathbb{T}_{m})-\diver{ }u_{m,t}\theta_{m}\,,\\
        \frac{\partial\theta_{m}}{\partial\nu}=0, \quad\quad \theta_{m}(0)=\theta_{n0}\,.
\end{align}
Convergence \eqref{convergence_nk} implies that the sequences $\diver u_{m,t}$ and $\textrm{T}_{n}(\textrm{G}(\theta_{m},\mathbb{T}_{m}):\mathbb{T}_{m})$ are bounded in  $L^{\infty}((0,\mathfrak{T})\times\Omega)$. As a result, by standard energy estimates we obtain the uniform bounds
$\theta_{m}$ in $L^{2}(0,\mathfrak{T};H^{1}(\Omega;\mathbb{R}))$ and
$\theta_{m,t}$ in $L^{2}((0,\mathfrak{T})\times\Omega)$. In consequence of the Aubin-Lions theorem, the following strong convergence in $L^{2}((0,\mathfrak{T})\times\Omega)$ holds $\theta_{m}\underset{m\to\infty}{\to}\tilde{\theta}$. By assumption \eqref{convergence_nk} once more, we notice that $\norm{\textrm{G}(\theta_{m},\mathbb{T}_{m})}_{L^{2}((0,\mathfrak{T})\times\Omega)}\leq C$. Therefore, the sequence $\{\textrm{G}(\theta_{m},\mathbb{T}_{m})\}_{m=1}^{\infty}$ admits a weakly convergent subsequence. By the uniqueness of the limit, the continuity of $\textrm{G}$, and the pointwise convergence of $\theta_m \to \tilde{\theta}$ and $\mathbb{T}_m \to \mathbb{T}$, the weak limit must be $\textrm{G}(\tilde{\theta},\mathbb{T})\in L^{2}((0,\mathfrak{T}\times\Omega)$. Moreover, using the dominated convergence theorem, this convergence can be upgraded from weak to strong in $L^2((0,\mathfrak{T})\times\Omega)$. Passing to the limit in the weak formulation of \eqref{theta_m_A_cont} and comparing it with the weak formulation for $\theta$, we obtain $\tilde{\theta}=\theta$. We are now in a position to improve the convergence of the sequence $\theta_{m}$ to $L^{2}(0,\mathfrak{T};H^{1}(\Omega))$.

We set $\psi_{m}=\theta(u,\mathbb{T})-\theta_{m}(u_{m},\mathbb{T}_{m}):=\theta-\theta_{m}$. Writing the energy balance \eqref{theta_nk-def_sol} for both $\theta$ and $\theta_{m}$, and subtracting the resulting relations, we obtain for every $\phi\in H^{1}(\Omega;\mathbb{R})$
\begin{align}
\label{psi_m_est_0}
    \int\limits_{\Omega}\psi_{m,t}\phi\,\di x
	&+\int\limits_{\Omega}\nabla\psi_{m} \nabla \phi\,\di x
	=-\int\limits_{\Omega}\theta(\diver{ } u_{t}-\diver{ } u_{m,t})\phi\, \di x
    -\int\limits_{\Omega}\psi_{m}\diver{ } u_{m,t}\phi\, \di x\nonumber\\&
	+\int\limits_{\Omega}\Big(\textrm{T}_{n}\big( \textrm{G}(\theta,\mathbb{T}):\mathbb{T}\big)-\textrm{T}_{n}\big( \textrm{G}(\theta_{m},\mathbb{T}_{m}):\mathbb{T}_{m}\big)\Big)\phi\,\di x\,.
\end{align}
We take $\phi=\psi_{m}\in H^{1}(\Omega;\mathbb{R})$ in \eqref{psi_m_est_0} and get after using H\"older's inequality
    \begin{align}\label{psi_m_est}
    &\frac{1}{2}\frac{\di{ }}{\di t}\int\limits_{\Omega}\abs{\psi_{m}}^{2}\di x
	+\int\limits_{\Omega}\abs{\nabla\psi_{m}}^{2}\di x
	=-\int\limits_{\Omega}\theta(\diver{ } u_{t}-\diver{ } u_{m,t})\psi_{m} \di x
    -\int\limits_{\Omega}\diver{ } u_{m,t}\psi_{m}^{2} \di x\nonumber\\&
	+\int\limits_{\Omega}\Big(\textrm{T}_{n}\big( \textrm{G}(\theta,\mathbb{T}):\mathbb{T}\big)-\textrm{T}_{n}\big( \textrm{G}(\theta,\mathbb{T}):\mathbb{T}_{m}\big)\Big)\psi_{m}\di x\nonumber\\&
    +\int\limits_{\Omega}\Big(\textrm{T}_{n}\big( \textrm{G}(\theta,\mathbb{T}):\mathbb{T}_{m}\big)-\textrm{T}_{n}\big( \textrm{G}(\theta_{m},\mathbb{T}_{m}):\mathbb{T}_{m}\big)\Big)\psi_{m}\di x\nonumber\\&\leq
    \norm{\theta}_{L^{2}(\Omega)}\norm{\diver{ } u_{t}-\diver{ } u_{m,t}}_{L^{\infty}(\Omega)}\norm{\psi_{m}}_{L^{2}(\Omega)}
    +\norm{\diver{ } u_{m,t}}_{L^{\infty}(\Omega)}\norm{\psi_{m}}_{L^{2}(\Omega)}^{2}\nonumber\\&
    +\int\limits_{\Omega}\abs{\textrm{G}(\theta,\mathbb{T})}\abs{\mathbb{T}-\mathbb{T}_{m}}\abs{\psi_{m}}\di x
    +\int\limits_{\Omega}\abs{\textrm{G}(\theta,\mathbb{T})-\textrm{G}(\theta_{m},\mathbb{T}_{m})}\abs{\mathbb{T}_{m}}\abs{\psi_{m}}\di x\,.
    \end{align}
Using Gronwall's lemma, we have 
    \begin{align*}
    \norm{\psi_{m}}^2_{L^{\infty}(0,\mathfrak{T};L^{2}(\Omega))}
    +\norm{\nabla\psi_{m}}^2_{L^{2}((0,\mathfrak{T})\times\Omega)}
    &\leq C\big(
    \norm{\diver{ } u_{t}-\diver{ } u_{m,t}}^2_{L^{\infty}((0,\mathfrak{T})\times\Omega)}\nonumber\\&\quad+\norm{\mathbb{T}-\mathbb{T}_{m}}^2_{L^{\infty}((0,\mathfrak{T})\times\Omega)}
    \nonumber\\&\quad+\norm{\textrm{G}(\theta,\mathbb{T})-\textrm{G}(\theta_{m},\mathbb{T}_{m})}^2_{L^{2}((0,\mathfrak{T})\times\Omega)}\big)\,,
    \end{align*}
    so finally
    \begin{align}\label{psi_m_est_final}
    \norm{\psi_{m}}^2_{L^{2}(0,\mathfrak{T};H^{1}(\Omega))}&\leq C\big(
    \norm{\diver{ } u_{t}-\diver{ } u_{m,t}}^2_{L^{\infty}((0,\mathfrak{T})\times\Omega)}\nonumber\\&\quad+\norm{\mathbb{T}-\mathbb{T}_{m}}^2_{L^{\infty}((0,\mathfrak{T})\times\Omega)}
    +\norm{\textrm{G}(\theta,\mathbb{T})-\textrm{G}(\theta_{m},\mathbb{T}_{m})}^2_{L^{2}((0,\mathfrak{T})\times\Omega)}\big)\,.
    \end{align} 
Now, we define $(a_{u}^{m},a_{\mathbb{T}}^{m}):=\mathcal{A}(u,\mathbb{T})-\mathcal{A}(u_{m},\mathbb{T}_{m})$ and write both the momentum equation and the inelastic constitutive equation for these differences.
    \begin{align}
    \label{operator}
	\int\limits_{\Omega} a_{u,tt}^{m}\varphi\,\di x 
	+ \int\limits_{\Omega}a_{\mathbb{T}}^{m}:\varepsilon(\varphi)\,\di x
	=&\int\limits_{\Omega}\psi_{m}\diver{ } \varphi\,\di x\,, \qquad \forall_{\varphi\in \mathcal{W}^{n}}\,,
	\end{align}
	\begin{align}
     \label{operator_1}
	\int\limits_{\Omega} \mathbb{C}^{-1}a_{\mathbb{T},t}^{m}: \psi\,\di x
	+\int\limits_{\Omega}\Big(\textrm{G}(\theta,\mathbb{T})-\textrm{G}(\theta_{m},\mathbb{T}_{m})\Big):\psi\, \di x
	=\int\limits_{\Omega}\varepsilon(a_{u,t}^{m}) : \psi\,\di x\,, \quad \forall_{\psi\in \mathcal{V}^{k}}\,.
	\end{align}
    We take $\varphi=a_{u,t}^{m}$ in \eqref{operator} and $\psi=a_{\mathbb{T}}^{m}$ in \eqref{operator_1}, which are proper test functions since $\mathcal{A}(u,\mathbb{T}),\mathcal{A}(u_{m},\mathbb{T}_{m})\in X_{nk}$, and get
    \begin{align*}
	\frac{1}{2}\frac{\di{ }}{\di t}\int\limits_{\Omega} \abs{a_{u,t}^{m}}^{2}\di x 
	+ \int\limits_{\Omega}a_{\mathbb{T}}^{m}:\varepsilon(a_{u,t}^{m})\,\di x
	=&\int\limits_{\Omega}\psi_{m}\diver{ } a_{u,t}^{m}\,\di x\,,
	\end{align*}
	\begin{align*}
	\frac{1}{2}\frac{\di{ }}{\di t}\int\limits_{\Omega} \mathbb{C}^{-1}a_{\mathbb{T}}^{m}: a_{\mathbb{T}}^{m}\,\di x
	+\int\limits_{\Omega}\Big(\textrm{G}(\theta,\mathbb{T})-\textrm{G}(\theta_{m},\mathbb{T}_{m})\Big):a_{\mathbb{T}}^{m}\, \di x
	=\int\limits_{\Omega}\varepsilon(a_{u,t}^{m}) :a_{\mathbb{T}}^{m}\,\di x\,. 
	\end{align*}
We integrate by parts, sum both above equations, and obtain
\begin{align*}
	&\frac{1}{2}\frac{\di{ }}{\di t}\norm{a_{u,t}^{m}}_{L^{2}(\Omega)}^{2}
    +\frac{1}{2}\frac{\di{ }}{\di t}\int\limits_{\Omega} \mathbb{C}^{-1}a_{\mathbb{T}}^{m}: a_{\mathbb{T}}^{m}\,\di x\nonumber\\&
	= -\int\limits_{\Omega}\nabla\psi_{m}a_{u,t}^{m}\,\di x
    -\int\limits_{\Omega}\Big(\textrm{G}(\theta,\mathbb{T})-\textrm{G}(\theta_{m},\mathbb{T}_{m})\Big):a_{\mathbb{T}}^{m}\, \di x\nonumber\\&
    \leq \norm{\nabla \psi_{m}}_{L^{2}(\Omega)}\norm{a_{u,t}^{m}}_{L^{2}(\Omega)}
    +\norm{\textrm{G}(\theta,\mathbb{T})-\textrm{G}(\theta_{m},\mathbb{T}_{m})}_{L^{2}(\Omega)}\norm{a_{\mathbb{T}}^{m}}_{L^{2}(\Omega)}\nonumber\\&
    \leq C_1\Big(\norm{\nabla \psi_{m}}^2_{L^{2}(\Omega)}+\norm{\textrm{G}(\theta,\mathbb{T})-\textrm{G}(\theta_{m},\mathbb{T}_{m})}^2_{L^{2}(\Omega)}\Big)\nonumber\\&\quad+C_2\big(\norm{a_{u,t}^{m}}^2_{L^{2}(\Omega)}    +\int\limits_{\Omega} \mathbb{C}^{-1}a_{\mathbb{T}}^{m}: a_{\mathbb{T}}^{m}\di x\big)\,,
\end{align*}
where in the last inequality we have employed the positive definiteness of the operator $\mathbb{C}^{-1}$, while the constants $C_{1}$ and $C_{2}$ do not depend on $m$. Then, from \eqref{psi_m_est_final} and Gronwall's lemma we get $\norm{(a_{u}^{m},a_{\mathbb{T}}^{m})}_{X_{nk}}\underset{m\to\infty}{\to}0$, showing that the operator $\mathcal{A}$ is continuous.\\
  
\textbf{Compactness of $\mathcal{A}$ }: we take $(u,\mathbb{T})\in X_{nk}$ such that $\norm{(u,\mathbb{T})}_{X_{nk}}\leq R$, where $R>0$ is fixed. We will show that $\overline{\mathcal{A}(u,\mathbb{T})}$ is compact. We set as previously $\mathcal{A}(u,\mathbb{T}):=(\tilde{u},\tilde{\mathbb{T}})$. \\
    Let $\theta$ be a solution to the balance of energy for $(u,\mathbb{T})$, i.e.
     \begin{align}
     \label{comp_1}
    \int\limits_{\Omega} \tilde{u}_{tt}\varphi\,\di x 
	+ \int\limits_{\Omega}\tilde{\mathbb{T}}:\varepsilon(\varphi)\,\di x
	-&\int\limits_{\Omega}\theta\diver{ } \varphi\,\di x
	= \int\limits_{\Omega} f_{n}\varphi\, \di x\,, \quad \forall_{\varphi\in\mathcal{W}^{n}}\,,
    \end{align}
    \begin{align}
    \label{comp_2}
    \int\limits_{\Omega} \mathbb{C}^{-1}\tilde{\mathbb{T}}_{t}: \psi\,\di x
	+\int\limits_{\Omega}\textrm{G}(\theta,\tilde{\mathbb{T}}):\psi \,\di x
	=\int\limits_{\Omega}\varepsilon(\tilde{u_{t}}) :\psi\,\di x\,, \quad \forall_{\psi\in\mathcal{V}^{k}}\,.
    \end{align}
We choose the functions in \eqref{comp_1} and \eqref{comp_2} to be $\varphi=\tilde{u}_{tt}$ and $\psi=\tilde{\mathbb{T}}_{t}$, respectively, as test functions. These are admissible since, by the Carath\'eodory theorem, they exist almost everywhere on $[0,\mathfrak{T}]$ and belong to the spaces $\mathcal{W}^{n}$ and $\mathcal{V}^{k}$ with respect to the spatial variable. Consequently, we obtain
\begin{align}
\label{comp_3}
    \int\limits_{\Omega} \abs{\tilde{u}_{tt}}^{2}\,\di x 
	+ \int\limits_{\Omega}\tilde{\mathbb{T}}:\varepsilon(\tilde{u}_{tt})\,\di x
	-&\int\limits_{\Omega}\theta\diver{ } \tilde{u}_{tt}\,\di x
	= \int\limits_{\Omega} f_{n}\tilde{u}_{tt} \,\di x\,,
\end{align}
\begin{align}
\label{comp_4}
    \int\limits_{\Omega} \mathbb{C}^{-1}\tilde{\mathbb{T}}_{t}: \tilde{\mathbb{T}}_{t}\,\di x
	+\int\limits_{\Omega}\textrm{G}(\theta,\tilde{\mathbb{T}}):\tilde{\mathbb{T}}_{t} \,\di x
	=\int\limits_{\Omega}\varepsilon(\tilde{u_{t}}) :\tilde{\mathbb{T}}_{t}\,\di x\,.
\end{align}
Summing equations \eqref{comp_3} and \eqref{comp_4} and then integrating over $(0,\mathfrak{T})$, we obtain
\begin{align*}
    &\norm{\tilde{u}_{tt}}_{L^{2}((0,\mathfrak{T})\times\Omega)}^{2}
    +\int\limits_{0}^{\mathfrak{T}}\int\limits_{\Omega} \mathbb{C}^{-1}\tilde{\mathbb{T}}_{t}: \tilde{\mathbb{T}}_{t}\,\di x \di t
    =\int\limits_{0}^{\mathfrak{T}}\int\limits_{\Omega} f_{n}\tilde{u}_{tt} \,\di x \di t
	+\int\limits_{0}^{\mathfrak{T}}\int\limits_{\Omega}\theta\diver{ } \tilde{u}
    _{tt}\,\di x \di t\nonumber \\ &\quad
    -\int\limits_{0}^{\mathfrak{T}}\int\limits_{\Omega}\tilde{\mathbb{T}}:\varepsilon(\tilde{u}_{tt})\,\di x\di t
    -\int\limits_{0}^{\mathfrak{T}}\int\limits_{\Omega}\textrm{G}(\theta,\tilde{\mathbb{T}}):\tilde{\mathbb{T}}_{t} \,\di x \di t
    +\int\limits_{0}^{\mathfrak{T}}\int\limits_{\Omega}\varepsilon(\tilde{u_{t}}) :\tilde{\mathbb{T}}_{t}\,\di x \di t\,.
\end{align*}
From the integration by parts, we have
\begin{align*}
    \int\limits_{0}^{\mathfrak{T}}\int\limits_{\Omega}\theta\diver{ } \tilde{u}
    _{tt}\,\di x \di t
    =- \int\limits_{0}^{\mathfrak{T}}\int\limits_{\Omega}\nabla\theta\tilde{u}
    _{tt}\,\di x \di t
    \leq \norm{\nabla\theta}_{L^{2}((0,\mathfrak{T})\times\Omega)}\norm{\tilde{u}
    _{tt}}_{L^{2}((0,\mathfrak{T})\times\Omega)}\,,
\end{align*}
    \begin{align*}
        &-\int\limits_{0}^{\mathfrak{T}}\int\limits_{\Omega}\tilde{\mathbb{T}}:\varepsilon(\tilde{u}_{tt})\,\di x\di t=
        \int\limits_{0}^{\mathfrak{T}}\int\limits_{\Omega}\tilde{\mathbb{T}}_{t}:\varepsilon(\tilde{u}_{t})\,\di x\di t
        -\int\limits_{\Omega}\tilde{\mathbb{T}}(t):\varepsilon(\tilde{u}_{t}(t))\,\di x\nonumber\\&\quad
        +\int\limits_{\Omega}\tilde{\mathbb{T}}(0):\varepsilon(\tilde{u}_{t}(0))\,\di x
        \leq \norm{\tilde{\mathbb{T}}_{t}}_{L^{2}((0,\mathfrak{T})\times\Omega)}\norm{\varepsilon(\tilde{u}_{t})}_{L^{2}((0,\mathfrak{T})\times\Omega)}\nonumber\\& \quad
        +\norm{\tilde{\mathbb{T}}}_{L^{\infty}(0,\mathfrak{T};L^{2}(\Omega))}\norm{\varepsilon(\tilde{u}_{t})}_{L^{1}(0,\mathfrak{T};L^{2}(\Omega))}
        +\norm{\tilde{\mathbb{T}}(0)}_{L^{2}(\Omega)}\norm{\varepsilon(\tilde{u}_{t}(0))}_{L^{2}(\Omega)}\,.
    \end{align*}

    Finally, we obtain
    \begin{align*}
    &(1-\varepsilon_{1}+\varepsilon_{2})\norm{\tilde{u}_{tt}}_{L^{2}((0,\mathfrak{T})\times\Omega)}^{2}+(1-\varepsilon_{3}-\varepsilon_{4})\norm{\tilde{\mathbb{T}}_{t}}_{L^{2}((0,\mathfrak{T})\times\Omega)}^{2} \leq 
    C(\varepsilon_{1})\norm{f_{n}}_{L^{2}((0,\mathfrak{T})\times\Omega)}^{2}\nonumber\\&\quad
    +C(\varepsilon_{2})\norm{\nabla\theta}_{L^{2}((0,\mathfrak{T})\times\Omega)}^{2}
    +C(\varepsilon_{3})\norm{\textrm{G}(\theta,\tilde{\mathbb{T}})}_{L^{2}((0,\mathfrak{T})\times\Omega)}^{2}
    +C(\varepsilon_{4})\norm{\varepsilon(\tilde{u}_{t})}_{L^{2}((0,\mathfrak{T})\times\Omega)}^{2}\nonumber\\&\quad
    +\norm{\tilde{\mathbb{T}}}_{L^{\infty}(0,\mathfrak{T};L^{2}(\Omega))}\norm{\varepsilon(\tilde{u}_{t})}_{L^{1}(0,\mathfrak{T};L^{2}(\Omega))}
    +\norm{\tilde{\mathbb{T}}(0)}_{L^{2}(\Omega)}\norm{\varepsilon(\tilde{u}_{t}(0))}_{L^{2}(\Omega)}\,.
    \end{align*}
    Because of \eqref{theta_nk_bou}, we obtain that $\norm{\nabla\theta}_{L^{2}((0,\mathfrak{T})\times\Omega)}^{2}\leq \tilde{C}(R)$. The fact that all norms in $X_{nk}$ are equivalent, we get that the third, fourth, and fifth terms on the right-hand side are bounded. Moreover, the initial conditions are the projections onto $\mathcal{W}^{n}$ and $\mathcal{V}^{k}$, we finally obtain that
    \begin{align*}
    &\norm{\tilde{u}_{tt}}_{L^{2}((0,\mathfrak{T})\times\Omega)}+\norm{\tilde{\mathbb{T}}_{t}}_{L^{2}((0,\mathfrak{T})\times\Omega)}\leq C(R)
    \end{align*}
    what can be rewritten as
    \begin{align*}
    &\norm{\tilde{u}_{t}}_{H^{1}(0,\mathfrak{T};L^{2}(\Omega))}+\norm{\tilde{\mathbb{T}}}_{H^{1}(0,\mathfrak{T};L^{2}(\Omega))} \leq C(R)\,.
    \end{align*}
By the compact embedding $H^{1}(0,\mathfrak{T}) \hookrightarrow\hookrightarrow C([0,\mathfrak{T}])$, the operator $\mathcal{A}$ maps bounded sets in $(u,\mathbb{T})$ to relatively compact sets. Hence, $\mathcal{A}$ is compact.\\

\textbf{Boundedness of the associated set: } let us set 
$$Y=\big\{ (u,\mathbb{T})\in X_{nk}:\exists_{\lambda\in(0,1]} \quad\mathcal{A}(u,\mathbb{T})=\frac{1}{\lambda}(u,\mathbb{T}) \big\}\,.$$
We take $(u,\mathbb{T})\in Y$ and $\lambda\in (0,1]$ such that $\frac{1}{\lambda}u=\mathcal{A}u$, $\frac{1}{\lambda}\mathbb{T}=\mathcal{A}\mathbb{T}$. Then $\big( (\frac{1}{\lambda}u,\frac{1}{\lambda}\mathbb{T}),\theta(u,\mathbb{T})\big)$ is a solution of the approximate system \eqref{approx_nk_1}-\eqref{approx_nk_3}. The energy balance for this system can be written as
    \begin{align*}
    &\frac{1}{2} \int\limits_{\Omega}\abs{\frac{1}{\lambda}u_{t}(t)}^{2}\di x
    +\frac{1}{2}\int\limits_{\Omega}\mathbb{C}^{-1}\big(\frac{1}{\lambda}\mathbb{T}(t)\big):\big(\frac{1}{\lambda}\mathbb{T}(t)\big)\,\di x
    +\int\limits_{\Omega} \theta(t)\, \di x\nonumber\\&
    =\frac{1}{2} \int\limits_{\Omega}\abs{\frac{1}{\lambda}u_{t}(0)}^{2}\di x
    +\frac{1}{2}\int\limits_{\Omega}\mathbb{C}^{-1}\big(\frac{1}{\lambda}\mathbb{T}(0)\big):\big(\frac{1}{\lambda}\mathbb{T}(0)\big)\,\di x
    +\int\limits_{\Omega} \theta(0) \di x
    +\int\limits_{0}^{t}\int\limits_{\Omega}f_{n}\frac{1}{\lambda}u_{t}\, \di x \di t\nonumber\\&
    =\frac{1}{2 \lambda^{2}} \int\limits_{\Omega}\abs{P_{\mathcal{W}^{n}}u_{1}}^{2}\,\di x
    +\frac{1}{2 \lambda^{2}}\int\limits_{\Omega}\mathbb{C}^{-1}P_{\mathcal{V}^{n}}\mathbb{T}_{0}:P_{\mathcal{V}^{n}}\mathbb{T}_{0}\,\di x
    +\int\limits_{\Omega} \theta_{0n}\, \di x
    +\int\limits_{0}^{t}\int\limits_{\Omega}f_{n}\frac{1}{\lambda}u_{t} \,\di x \di t \nonumber\\&
    \leq
    \frac{1}{2 \lambda^{2}} \int\limits_{\Omega}\abs{P_{\mathcal{W}^{n}}u_{1}}^{2}\,\di x
    +\frac{1}{2 \lambda^{2}}\int\limits_{\Omega}\mathbb{C}^{-1}P_{\mathcal{V}^{n}}\mathbb{T}_{0}:P_{\mathcal{V}^{n}}\mathbb{T}_{0}\,\di x
    +\int\limits_{\Omega} \theta_{0n}\, \di x\nonumber\\&\quad
    +C(\varepsilon)\norm{f_{n}}_{L^{2}((0,\mathfrak{T})\times\Omega)}^{2}+\frac{\varepsilon}{\lambda^{2}}\norm{u_{t}}_{L^{\infty}(0,\mathfrak{T};L^{2}(\Omega))}^{2}
    \end{align*}
    for any $\varepsilon>0$, then we get the following inequality
    \begin{align*}
    &\frac{1-\varepsilon}{2\lambda^{2}} \norm{u_{t}}_{L^{\infty}(0,\mathfrak{T};L^{2}(\Omega))}^{2}
    +\frac{C}{2 \lambda^{2}}\norm{\mathbb{T}}_{L^{\infty}(0,\mathfrak{T};L^{2}(\Omega))}^{2}  
    \leq
    \frac{1}{2 \lambda^{2}} \int\limits_{\Omega}\abs{P_{\mathcal{W}^{n}}u_{1}}^{2}\di x\nonumber\\& 
    +\frac{1}{2 \lambda^{2}}\int\limits_{\Omega}\mathbb{C}^{-1}P_{\mathcal{V}^{n}}\mathbb{T}_{0}:P_{\mathcal{V}^{n}}\mathbb{T}_{0}\di x
    +\int\limits_{\Omega} \theta_{0n} \di t
    +C(\varepsilon)\norm{f_{n}}_{L^{2}((0,\mathfrak{T})\times\Omega)}^{2}
    \end{align*}
    which implies that $\norm{(u,\mathbb{T})}_{X_{nk}}\leq C(n)$, so $Y$ is bounded.\\

The considerations in this subsection so far lead to the following existence theorem for the approximated system \eqref{approx_nk_1}-\eqref{approx_nk_3}.    
\begin{theorem}
Let the assumptions of Theorem \ref{TWIERDZENIE} be satisfied. Then, for each $n,k \in \mathbb{N}$, there exists a solution $(u_{nk},\mathbb{T}_{nk},\theta_{nk})$ to the approximate problem \eqref{approx_nk_1}-\eqref{approx_nk_3}. Moreover, for each $n,k \in \mathbb{N}$, the component $\theta_{nk}$ satisfies $\theta_{nk}\in L^{2}(0,\mathfrak{T};H^{2}(\Omega;\mathbb{R}))\cap H^{1}(0,\mathfrak{T};L^{2}(\Omega;\mathbb{R}))$.
\end{theorem}
\begin{proof} 
From Sch\"afer's fixed point theorem, there exists a fixed point $(u,\mathbb{T})$ of the mapping $\mathcal{A}$. The construction of the operator $\mathcal{A}$ implies that for each $n,k \in \mathbb{N}$, there exists a solution $(u_{nk},\mathbb{T}_{nk},\theta_{nk})$ to the approximate problem \eqref{approx_nk_1}-\eqref{approx_nk_3} possessing the regularity specified in \eqref{reg_apprx}. Additionally, we note that the sequence
 \begin{align*}
-\theta_{nk}\diver{ }u_{nk,t}+\textrm{T}_{n}\big(\textrm{G}(\theta_{nk},\mathbb{T}_{nk}):\mathbb{T}_{nk}\big) 
\end{align*}
is bounded in $L^{2}((0,\mathfrak{T})\times\Omega)$. Hence, standard regularity theory applied to equation \eqref{approx_nk_3} implies that  for each $n,k \in \mathbb{N}$, $\theta_{nk}$ belongs to the space $L^{2}(0,\mathfrak{T};H^{2}(\Omega;\mathbb{R}))\cap H^{1}(0,\mathfrak{T};L^{2}(\Omega;\mathbb{R}))$.
\end{proof} 
At the end of this subsection, we prove the following lemma, which is similar to \cite{CieslakMuha_3D}, from which the positivity of the temperature at each approximation step will follow. 
\begin{lemma}\label{positivity_of_theta}
Let $a \in L^{\infty}((0,\mathfrak{T}) \times \Omega)$, $b \in L^{2}((0,\mathfrak{T}) \times \Omega)$ with $b \ge 0$, and $\theta_0 \in H^1(\Omega;\mathbb{R})$ satisfy $\theta_0 > 0$ almost everywhere in $\Omega$. Let $\theta \in L^2(0,\mathfrak{T}; H^2(\Omega;\mathbb{R})) \cap H^1(0,\mathfrak{T}; L^2(\Omega;\mathbb{R}))$ be the solution to the system
\begin{align}\label{positive_theta_system_1}
    \theta_t - \Delta \theta &= -a \theta + b\,, && \text{in } (0,\mathfrak{T}) \times \Omega\,, \nonumber\\
    \frac{\partial \theta}{\partial \nu} &= 0\,, && \text{on } (0,\mathfrak{T}) \times \partial\Omega\,, \\
    \theta(0) &= \theta_0\,,&& \text{in } \Omega\,. \nonumber
\end{align}
Then $\theta$ remains strictly positive for all $t \in [0,\mathfrak{T}]$ and almost every $x \in \Omega$, and in particular,
\begin{align*}
    \theta(t,x) \ge \left( \min_{x \in \Omega} \theta_0(x) \right) \exp\Bigg(- \int_0^t \|a(s)\|_{L^{\infty}(\Omega)} \, ds \Bigg) > 0\,.
\end{align*}
\end{lemma}
 \begin{proof} 
We set 
$$C_{0}=\min\limits_{x\in\Omega}\theta_{0}(x)\quad \textrm{and}\quad d(t,x)=\theta(t,x)-C_{0}e^{-\int\limits_{0}^{t}\norm{a(s)}_{L^{\infty}(\Omega)}\di s}\,.$$
Rewriting system \eqref{positive_theta_system_1} in terms of $d$ gives
\begin{align}\label{positive_theta_system_2}
        &d_{t}-\Delta d+ad=C_{0}e^{-\int\limits_{0}^{t}\norm{a(s)}_{L^{\infty}(\Omega)}\di s}\big(\norm{a(t)}_{L^{\infty}(\Omega)}-a \big)+b\,,\quad\text{ in } (0,\mathfrak{T})\times\Omega,\nonumber\\
        &\frac{\partial d}{\partial\nu}=0\,, \quad \quad \quad \quad \quad \quad \quad \quad \quad  \quad \quad \quad \quad \quad \quad \quad \quad \quad \quad \quad  \quad \quad\, \text{ on } (0,\mathfrak{T})\times\partial\Omega\,,\\
        &d(0)=\theta_{0}-C_{0}\geq 0\,, \qquad \qquad \qquad \qquad \qquad \qquad \quad \quad \quad  \quad \quad\,\,\, \text{in } \Omega\,.\nonumber
\end{align}
From \eqref{positive_theta_system_2}$_1$\,, it follows that
        \begin{align}\label{d_inequality}
            d_{t}-\Delta d+ad\geq 0\,.
        \end{align}     
Multiplying \eqref{d_inequality} by the negative part $d^- = \max\{0, -d\}$, integrating over $\Omega$ and $(0,t)$ for $t < \mathfrak{T}$, and observing that $d^-(0) = 0$, we obtain
        \begin{align*}
            \int\limits_{0}^{t}\int\limits_{\Omega}d_{t}^{-}d^{-} \,\di x \di t +\int\limits_{0}^{t}\int\limits_{\Omega}\abs{\nabla d^{-}}^{2}\,\di x \di t
            \leq 
            -\int\limits_{0}^{t}\int\limits_{\Omega}a\abs{\nabla d^{-}}^{2}\,\di x \di t\,,
        \end{align*}
which gives
\begin{align*}
            \frac{1}{2}\norm{d^{-}(t)}_{L^{2}(\Omega)}^{2}
            \leq \frac{1}{2}\norm{d^{-}(0)}_{L^{2}(\Omega)}^{2}
            +C\int\limits_{0}^{t}\norm{d^{-}(s)}^{2}\di t\,.
\end{align*}
It follows from Gronwall's lemma that $d(t,x) \ge 0$ for almost every $x \in \Omega$ and all $t \in [0,\mathfrak{T}]$. Consequently,
\[
\theta(t,x) \ge C_0 \, e^{-\int_0^t \|a(s)\|_{L^\infty(\Omega)} \, ds} > 0
\quad \text{for almost every } x \in \Omega \text{ and all } t \in [0,\mathfrak{T}]\,.
\]
\end{proof} 
\begin{remark}
It follows from Lemma \ref{positivity_of_theta} that, for each $n,k \in \mathbb{N}$, 
\[
\theta_{nk}(t,x) > 0 \quad \text{for almost every } x \in \Omega \text{ and all } t \in [0,\mathfrak{T}].
\]
\end{remark}
\section{Limit passage in $k$ for fixed $n$}
In this section, for fixed $n\in\mathbb{N}$, we aim to pass to the limit $k\to\infty$ in equations \eqref{approx_nk_1}-\eqref{approx_nk_3}. This is one of the crucial steps in the proof of the main result of this paper, since after taking the limit $k\to\infty$, the inelastic constitutive relation \eqref{approx_nk_2} will be satisfied pointwise almost everywhere for $(t,x)\in (0,\mathfrak{T})\times\Omega$
 for each $n\in\mathbb{N}$. We therefore begin with deriving certain bounds for the solutions to the auxiliary problem \eqref{approx_nk_1}-\eqref{approx_nk_3}, which will allow us to choose appropriate subsequences and pass to the limit with respect to the parameter $k\to\infty$.
\begin{proposition}\label{Prop_est_fixed_n}
Assume that the hypotheses of Theorem~\ref{TWIERDZENIE} hold. 
Then, for every $\mathfrak{T}>0$ there exists a constant $C(\mathfrak{T})>0$, 
independent of $k$ (and also of $n$), such that the approximate solutions 
$(u_{nk},\mathbb{T}_{nk},\theta_{nk})$ satisfy the uniform estimate
\begin{align*}
&\norm{u_{nk,t}}_{L^{\infty}\!\left(0,\mathfrak{T};L^{2}(\Omega;\mathbb{R}^{3})\right)}
+\norm{\mathbb{T}_{nk}}_{L^{\infty}\!\left(0,\mathfrak{T};L^{2}(\Omega;\mathcal{S}^{3})\right)}
+\norm{\theta_{nk}}_{L^{\infty}\!\left(0,\mathfrak{T};L^{1}(\Omega;\R)\right)}
\leq C(\mathfrak{T})\,.
\end{align*}
\end{proposition}
\begin{proof}
First, we derive the energy balance corresponding to the system \eqref{approx_nk_1}-\eqref{approx_nk_3}. To achieve this, we choose the test functions $\varphi=u_{nk,t}$, $\psi=\mathbb{T}_{nk}$, and $\phi=1$, insert them into the respective equations, and then sum up the resulting identities. This yields, for all $t\leq\mathfrak{T}$, the following relation:
\allowdisplaybreaks
\begin{align*}
    &\frac{1}{2} \int\limits_{\Omega}\abs{u_{nk,t}(t)}^{2}\di x
    +\frac{1}{2}\int\limits_{\Omega}\mathbb{C}^{-1}\mathbb{T}_{nk}(t):\mathbb{T}_{nk}(t)\,\di x
    +\int\limits_{\Omega}\theta_{nk}(t)\, \di x
    +\int\limits_{0}^{t}\int\limits_{\Omega}\textrm{G}(\theta_{nk},\mathbb{T}_{nk}):\mathbb{T}_{nk}\,\di x \di t
    \nonumber\\&
    =\frac{1}{2} \int\limits_{\Omega}\abs{u_{nk,t}(0)}^{2}\,\di x
    +\frac{1}{2}\int\limits_{\Omega}\mathbb{C}^{-1}\mathbb{T}_{nk}(0):\mathbb{T}_{nk}(0)\,\di x
    +\int\limits_{\Omega} \theta_{nk}(0)\,\di x\nonumber\\&\quad
    +\int\limits_{0}^{t}\int\limits_{\Omega}\textrm{T}_{n}\big(\textrm{G}(\theta_{nk},\mathbb{T}_{nk}):\mathbb{T}_{nk}\big)\,\di x \di t
    +\int\limits_{0}^{t}\int\limits_{\Omega}f_{n}u_{nk,t}\, \di x \di t\nonumber\\&
    =\frac{1}{2} \int\limits_{\Omega}\abs{P_{\mathcal{W}^{n}}u_{1}}^{2}\,\di x
    +\frac{1}{2}\int\limits_{\Omega}\mathbb{C}^{-1}P_{\mathcal{V}^{n}}\mathbb{T}_{0}:P_{\mathcal{V}^{n}}\mathbb{T}_{0}\,\di x
    +\int\limits_{\Omega} \theta_{0n}\, \di x \nonumber\\&\quad
    +\int\limits_{0}^{t}\int\limits_{\Omega}\textrm{T}_{n}\big(\textrm{G}(\theta_{nk},\mathbb{T}_{nk}):\mathbb{T}_{nk}\big)\,\di x \di t    
    +\int\limits_{0}^{t}\int\limits_{\Omega}f_{n}u_{nk,t} \,\di x \di t \nonumber\\&
    \leq
    \frac{1}{2} \int\limits_{\Omega}\abs{P_{\mathcal{W}^{n}}u_{1}}^{2}\,\di x
    +\frac{1}{2}\int\limits_{\Omega}\mathbb{C}^{-1}P_{\mathcal{V}^{n}}\mathbb{T}_{0}:P_{\mathcal{V}^{n}}\mathbb{T}_{0}\,\di x
    +\int\limits_{\Omega} \theta_{0n} \di x
    +\int\limits_{0}^{t}\int\limits_{\Omega}\textrm{G}(\theta_{nk},\mathbb{T}_{nk}):\mathbb{T}_{nk}\,\di x \di t\nonumber\\&\quad
    +C(\varepsilon)\norm{f_{n}}_{L^{2}((0,\mathfrak{T})\times\Omega)}^{2}+\varepsilon\norm{u_{nk,t}}_{L^{\infty}(0,\mathfrak{T};L^{2}(\Omega))}^{2}\nonumber\\&
    \leq 
    \frac{1}{2} \int\limits_{\Omega}\abs{u_{1}}^{2}\,\di x
    +\frac{1}{2}\int\limits_{\Omega}\mathbb{C}^{-1}\mathbb{T}_{0}:\mathbb{T}_{0}\,\di x
    +C\|\theta_{0}\|_{L^1(\Omega)}  +\int\limits_{0}^{t}\int\limits_{\Omega}\textrm{G}(\theta_{nk},\mathbb{T}_{nk}):\mathbb{T}_{nk}\,\di x \di t\nonumber\\&\quad
    +\tilde{C}(\varepsilon)\norm{f}_{L^{2}((0,\mathfrak{T})\times\Omega)}^{2}
    +\varepsilon\norm{u_{nk,t}}_{L^{\infty}(0,\mathfrak{T};L^{2}(\Omega))}^{2}
\end{align*}
for any $\varepsilon>0$, the constant $C>0$ is independent of both $k$ and $n$. We employ here the fact that the initial conditions for $u_{nk}$ and $\mathbb{T}_{nk}$ in the approximate system are given by orthogonal projections onto $\mathcal{W}^{n}$ and $\mathcal{V}^{n}$, respectively. Together with the convergence property \eqref{init_cond_nk_4} and \eqref{Property_f_n}, this yields - after taking the essential supremum - that the sequences are bounded by a constant independent of $k$ and even of $n$. This concludes the proof.
\end{proof} 
 \begin{remark}\label{Remark3.2}
Since $n\in\mathbb{N}$ is fixed, then the sequences
\begin{equation*}
\{\diver{ }u_{nk,t}\}_{k=1}^\infty\,, \qquad \big\{\mathrm{T}_{n}\big( \mathrm{G}(\theta_{nk},\mathbb{T}_{nk}):\mathbb{T}_{nk}\big)\big\}_{k=1}^\infty
\end{equation*}
are bounded in $L^{\infty}((0,\mathfrak{T})\times\Omega)$ with bounds independent of $k$ (but possibly depending on $n$).	Consequently, by the standard theory of parabolic equations, we obtain that
$$\norm{\theta_{nk}}_{L^{2}(0,\mathfrak{T};H^{1}(\Omega;\mathbb{R}))}\leq C(\mathfrak{T},n)\quad \textrm{and}\quad \|\theta_{nk,t}\|_{L^{2}(0,\mathfrak{T};L^{2}(\Omega;\mathbb{R}))}\leq C(\mathfrak{T},n)\,.$$        
\end{remark}        
We now turn to estimating the time derivatives of the sequences appearing in Proposition \ref{Prop_est_fixed_n}.  
  \begin{proposition}
Under the assumptions of Theorem \ref{TWIERDZENIE}, there exists a constant\\ $C(\mathfrak{T},n)>0$, independent of $k$, such that the following estimate is satisfied
\begin{align*}
        \norm{u_{nk,tt}}_{L^{2}((0,\mathfrak{T})\times\Omega)}
        +\norm{\mathbb{T}_{nk,t}}_{L^{2}((0,\mathfrak{T})\times\Omega)}
        \leq C(\mathfrak{T},n)\,.
\end{align*}
\end{proposition}       
\begin{proof} 
We take $\varphi = u_{nk,tt}$ and $\psi = \mathbb{T}_{nk,t}$ as test functions in \eqref{approx_nk_1} and \eqref{approx_nk_2}, respectively, integrate over $(0,\mathfrak{T})$, and then sum the resulting identities to obtain
\begin{align*}
	&\norm{u_{nk,tt}}_{L^{2}((0,\mathfrak{T})\times\Omega)}^{2}
    +\int\limits_{0}^{\mathfrak{T}}\int\limits_{\Omega} \mathbb{C}^{-1}\mathbb{T}_{nk,t}:\mathbb{T}_{nk,t}\, \di x \di t \nonumber \\ &
	=- \int\limits_{0}^{\mathfrak{T}}\int\limits_{\Omega}\mathbb{T}_{nk}:\varepsilon(u_{nk,tt})\,\di x \di t
	+\int\limits_{0}^{\mathfrak{T}}\int\limits_{\Omega}\theta_{nk}\diver{ } u_{nk,tt}\,\di x \di t
	+ \int\limits_{0}^{\mathfrak{T}}\int\limits_{\Omega} f_{n}u_{nk,tt} \,\di x \di t\\ &\quad
    -\int\limits_{0}^{\mathfrak{T}}\int\limits_{\Omega} \textrm{G}(\theta_{nk},\mathbb{T}_{nk}):\mathbb{T}_{nk,t}\, \di x \di t
    +\int\limits_{0}^{\mathfrak{T}}\int\limits_{\Omega} \varepsilon(u_{nk,t}):\mathbb{T}_{nk,t}\, \di x \di t\nonumber \,.
\end{align*}
Using integration by parts we have
\begin{align*}
    \int\limits_{0}^{\mathfrak{T}}\int\limits_{\Omega}\theta_{nk}\diver{ } u_{nk,tt}\,\di x \di t
    =-\int\limits_{0}^{\mathfrak{T}}\int\limits_{\Omega}\nabla\theta_{nk}u_{nk,tt}\,\di x \di t
\end{align*}
and
\begin{align*}
    &-\int\limits_{0}^{\mathfrak{T}}\int\limits_{\Omega}\mathbb{T}_{nk}:\varepsilon(u_{nk,tt})\,\di x \di t
    = \int\limits_{0}^{\mathfrak{T}}\int\limits_{\Omega}\mathbb{T}_{nk,t}:\varepsilon(u_{nk,t})\,\di x \di t
    -\int\limits_{\Omega}\mathbb{T}_{nk}(\mathfrak{T}):\varepsilon(u_{nk,t})(\mathfrak{T})\,\di x \nonumber\\&\quad
    +\int\limits_{\Omega}\mathbb{T}_{nk}(0):\varepsilon(u_{nk,t})(0)\,\di x
    \leq \norm{\mathbb{T}_{nk,t}}_{L^{2}((0,\mathfrak{T})\times\Omega)}\norm{\varepsilon(u_{nk,t})}_{L^{\infty}((0,\mathfrak{T})\times\Omega)} \nonumber\\&\quad
    +\norm{\mathbb{T}_{nk}}_{L^{\infty}(0,\mathfrak{T};L^{2}(\Omega))}\norm{\varepsilon(u_{nk,t})}_{L^{\infty}((0,\mathfrak{T})\times\Omega)}
    +\norm{P_{\mathcal{V}^{n}}\mathbb{T}_{0}}_{L^{2}(\Omega)}\norm{\varepsilon(P_{\mathcal{W}^{n}}u_{1})}_{L^{2}(\Omega)}
    \nonumber\\&
    \leq \norm{\mathbb{T}_{nk,t}}_{L^{2}((0,\mathfrak{T})\times\Omega)}\norm{\varepsilon(u_{nk,t})}_{L^{\infty}((0,\mathfrak{T})\times\Omega)} 
    +\norm{\mathbb{T}_{nk}}_{L^{\infty}(0,\mathfrak{T};L^{2}(\Omega))}\norm{\varepsilon(u_{nk,t})}_{L^{\infty}((0,\mathfrak{T})\times\Omega)} \nonumber\\&\quad
    +\norm{\mathbb{T}_{0}}_{L^{2}(\Omega)}\norm{\varepsilon(u_{1})}_{L^{2}(\Omega)}\,.
\end{align*}
Finally, we obtain
\begin{align*}
	&(1-\varepsilon_{1}-\varepsilon_{2})\norm{u_{nk,tt}}_{L^{2}((0,\mathfrak{T})\times\Omega)}^{2}
    +C(1-\varepsilon_{3}-\varepsilon_{4})\norm{\mathbb{T}_{nk,t}}_{L^{2}((0,\mathfrak{T})\times\Omega)}^{2}
    \nonumber \\ &
    \leq
    C(\varepsilon_{1})\norm{\nabla\theta_{nk}}_{L^{2}((0,\mathfrak{T})\times\Omega)}^{2}
    +C(\varepsilon_{2})\norm{f}_{L^{2}((0,\mathfrak{T})\times\Omega)}^{2}
    +C(\varepsilon_{3})\norm{\textrm{G}(\theta_{nk},\mathbb{T}_{nk})}_{L^{2}((0,\mathfrak{T})\times\Omega)}^{2} \nonumber \\ &\quad
    +C(\varepsilon_{4})\norm{\varepsilon(u_{nk,t})}_{L^{\infty}((0,\mathfrak{T})\times\Omega)}^{2}
	+\norm{\mathbb{T}_{nk}}_{L^{\infty}(0,\mathfrak{T};L^{2}(\Omega))}\norm{\varepsilon(u_{nk,t})}_{L^{\infty}((0,\mathfrak{T})\times\Omega)} \\&\quad
    +\norm{\mathbb{T}_{0}}_{L^{2}(\Omega)}\norm{\varepsilon(u_{1})}_{L^{2}(\Omega)}\nonumber\,,
\end{align*}
where $\varepsilon_{1}, \varepsilon_{2},\varepsilon_{3},\varepsilon_{4}>0$. Taking $1-\varepsilon_{1}-\varepsilon_{2}>0$ and $1-\varepsilon_{3}-\varepsilon_{4}>0$, and using the estimate from Proposition \ref{Prop_est_fixed_n} and Remark \ref{Remark3.2} together with the fact that $n$ is fixed, we complete the proof.
\end{proof}        
\subsection{Convergences for fixed $n$}
Having established the above estimates, we may pass to the limit with respect to $k$ for fixed $n$. As a result, we obtain a system of equations depending solely on the level $n$. In particular, the inelastic constitutive relation as well as the heat equation are then satisfied pointwise almost everywhere.\\
\begin{remark}\label{Remark_convergences_nk}
The uniform estimates with respect to $k\in\mathbb{N}$ derived in Subsection 3.1 ensure that, up to a subsequence, the following weak convergences hold.   
	\begin{align*}
	u_{nk,t}\underset{k\to\infty}{\rightharpoonup} u_{n,t}& \text{\quad in \quad}  L^{2}(0,\mathfrak{T};H_{0}^{1}(\Omega;\mathbb{R}^{3}))\,,\\
	u_{nk,tt}\underset{k\to\infty}{\rightharpoonup} u_{n,tt}& \text{\quad in \quad} L^{2}(0,\mathfrak{T};L^{2}(\Omega;\mathbb{R}^{3}))\,,\\
	\mathbb{T}_{nk}\underset{k\to\infty}{\rightharpoonup} \mathbb{T}_{n}& \text{\quad in \quad} L^{2}(0,\mathfrak{T};L^{2}(\Omega;\mathcal{S}^{3}))\,,\\
	\mathbb{T}_{nk,t}\underset{k\to\infty}{\rightharpoonup} \mathbb{T}_{n,t}& \text{\quad in \quad} L^{2}(0,\mathfrak{T};L^{2}(\Omega;\mathcal{S}^{3}))\,,\\
	\theta_{nk}\underset{k\to\infty}{\rightharpoonup} \theta_{n}& \text{\quad in \quad} L^{2}(0,\mathfrak{T};H^{1}(\Omega;\mathbb{R}))\,,\\
	\theta_{nk,t}\underset{k\to\infty}{\rightharpoonup} \theta_{n,t}& \text{\quad in \quad} L^{2}(0,\mathfrak{T};L^{2}(\Omega;\mathbb{R}))\,.
	\end{align*}
\end{remark}  
Below we observe that some of the above weak convergences in fact can be improved to strong ones in the appropriate spaces.
\begin{remark} \label{Remark_theta_nk_strong_convergence}
		From Aubin-Lions \cite{Roubicekbook} theorem, at least for a subsequence, the strong convergence holds
		\begin{align*}
		\theta_{nk}\underset{k\to\infty}{\to} \theta_{n}& \text{ \quad in \quad} L^{2}(0,\mathfrak{T};L^{2}(\Omega;\mathbb{R}))\,.
		\end{align*} 
\end{remark}
\begin{remark}\label{Remark_u_nk_strong_convergence}
Since $n\in\mathbb{N}$ is fixed, for almost every  $t\in[0,\mathfrak{T}]$ the vector $u_{nk,t}$ belongs to the finite-dimensional space $\mathcal{W}^{n}$. Consequently, the embeddings  $\mathcal{W}^{n} \hookrightarrow\hookrightarrow \mathcal{W}^{n} \hookrightarrow \mathcal{W}^{n} $ hold, where the first one is compact by the Bolzano-Weierstrass theorem. Hence, by the Aubin-Lions theorem, there exists a subsequence (not relabeled) such that the following strong convergence is satisfied
	\begin{align*}
		u_{nk,t}\underset{k\to\infty}{\to} u_{n,t} & \text{ \quad in \quad} L^{2}(0,\mathfrak{T};\mathcal{W}^{n})\,.
		\end{align*}
\end{remark}
\begin{remark}
		After a possible change of values on a set of measure zero, we obtain
		\begin{align*}
		u_{n,t}\in C([0,\mathfrak{T}];L^{2}(\Omega;\mathbb{R}^{3}))\,,\quad
		\mathbb{T}_{n}\in C([0,\mathfrak{T}];L^{2}(\Omega;\mathcal{S}^{3}))\,,\quad
        \theta_{n}\in C([0,\mathfrak{T}];L^{2}(\Omega;\mathbb{R}))\,.\\
		\end{align*}
\end{remark}
The next step is to pass to the limit as $k\to\infty$. The established weak and strong convergences allow us to handle all the linear terms in the system \eqref{approx_nk_1}-\eqref{approx_nk_3}. It remains, however, to analyze the nonlinear terms in \eqref{approx_nk_2}-\eqref{approx_nk_3} and identify their limits. In this analysis, we shall rely on calculations and results from the article \cite{OwczarekWielgos23}; due to the fixed level of approximation ($n$ being fixed), the passage to the limit with respect to $k$ is straightforward, and therefore the details are omitted here.	
\begin{proposition}
\label{weak_G(theta_k,T_k)}
Assume that the assumptions of Theorem~\ref{TWIERDZENIE} are satisfied and that the convergences stated in Remarks~\ref{Remark_convergences_nk}-\ref{Remark_u_nk_strong_convergence} hold. Then, the following weak convergence holds
\begin{align*}
	\mathrm{G}(\theta_{nk},\mathbb{T}_{nk})
	\underset{k\to\infty}{\rightharpoonup}
	\mathrm{G}(\theta_{n},\mathbb{T}_{n}) &\text{\quad in \quad} L^{2}(0,\mathfrak{T};L^{2}(\Omega;\mathcal{S}^{3}))\,.
	\end{align*}
\end{proposition}
\begin{proof} 
From property \eqref{bound_G} and the boundedness of 
$\{\mathbb{T}_{nk} \}_{k=1}^{\infty}$ in $L^{2}((0,\mathfrak{T})\times\Omega)$, we deduce that
 \begin{align*}
        \norm{\textrm{G}(\theta_{nk},\mathbb{T}_{nk})}_{L^{2}((0,\mathfrak{T})\times\Omega)}\leq C(\mathfrak{T})\,,
\end{align*}
so that $\textrm{G}(\theta_{nk},\mathbb{T}_{nk})\underset{k\to\infty}{\rightharpoonup}\gamma_{n}$ in $L^{2}((0,\mathfrak{T})\times\Omega)$. It remains to identify the weak limit. The argument follows the standard Minty method: we test the $nk$-system with $u_{nk,t}$ and $\mathbb{T}_{nk}$, compare it with the corresponding limit system tested by $u_{n,t}$ and $\mathbb{T}_{n}$, and then pass to the limit $k\to\infty$. Using weak semicontinuity of the $L^2$-norm and of the quadratic form induced by $\mathbb{C}^{-1}$, as well as the weak convergence $u_{nk,t}\underset{k\to\infty}{\rightharpoonup}u_{n,t}$ in $L^{2}(0,\mathfrak{T};H_{0}^{1}(\Omega))$, we deduce
\begin{equation}\label{limsupc_G(T):T_L1}
		\limsup_{k\to\infty}\int\limits_{0}^{\mathfrak{T}}\int\limits_{\Omega}\textrm{G}(\theta_{nk},\mathbb{T}_{nk}):\mathbb{T}_{nk} \,\di x\di t
		\leq\int\limits_{0}^{\mathfrak{T}}\int\limits_{\Omega}\gamma_{n}:\mathbb{T}_{n}\,\di x\di t\,,
\end{equation}
which provides the inequality required for Minty's trick. The detailed proof can be found in Proposition $4.13$ of \cite{OwczarekWielgos23}.
\end{proof}
Proposition \ref{weak_G(theta_k,T_k)} together with the convergences established in Remark \ref{convergence_nk} allows us to pass to the limit in equation \eqref{approx_nk_2} and obtain for a.e. $t\in [0,\mathfrak{T}]$
	\begin{align}\label{approx_n_2_0}
	&\int\limits_{\Omega} \mathbb{C}^{-1}\mathbb{T}_{n,t}:\psi\,\di x
	+\int\limits_{\Omega}\textrm{G}(\theta_{n}, \mathbb{T}_{n}):\psi\,\di x
	=\int\limits_{\Omega}\varepsilon(u_{n,t}):\psi\,\di x
	\end{align}
for all $\psi\in L^{2}(\Omega;\mathcal{S}^{3})$. What remains is to address the nonlinear terms related to the heat equation, which is the subject of the following propositions.
\begin{proposition}
 Assume that the assumptions of Theorem \ref{TWIERDZENIE} are satisfied and that the convergences stated in Remarks \ref{Remark_convergences_nk}-\ref{Remark_u_nk_strong_convergence} hold. Then we obtain the following convergence
\begin{align*}
    \theta_{nk}\diver{ }u_{nk,t}
	\underset{k\to\infty}{\to}
	\theta_{n}\diver{ }u_{n,t}
    &\text{\quad in \quad} L^{2}(0,\mathfrak{T};L^{2}(\Omega;\mathbb{R}))\,.
\end{align*}
\end{proposition}
\begin{proof}
By Remarks \ref{Remark_theta_nk_strong_convergence} and \ref{Remark_u_nk_strong_convergence}, we know that 
\(\displaystyle\theta_{nk} \underset{k\to\infty}{\to} \theta_{n}\) and \(\displaystyle\operatorname{div} u_{nk} \underset{k\to\infty}{\to} \operatorname{div} u_{n}\) strongly converge in 
\(L^{2}([0,\mathfrak{T}]\times\Omega)\).
\end{proof}
\begin{remark}
Using the pointwise convergence of $\{\theta_{nk}\}_{k=1}^{\infty}$ (at least for a subsequence), which follows from Remark \ref{Remark_theta_nk_strong_convergence}, one shows that for every fixed $\mathbb{T}\in L^{2}([0,\mathfrak{T}]\times\Omega)$,
\begin{equation}
\label{strong_G(theta_nk)}
\mathrm{G}(\theta_{nk},\mathbb{T}) \underset{k\to\infty}{\to} \mathrm{G}(\theta_{n},\mathbb{T}) \quad \text{in } \quad L^{2}([0,\mathfrak{T}]\times\Omega)\,.
\end{equation}
Together with the strong convergence \eqref{strong_G(theta_nk)} and the weak convergence 
$\mathbb{T}_{nk}\rightharpoonup \mathbb{T}_{n}$ in $L^{2}(0,\mathfrak{T};L^{2}(\Omega))$, 
the monotonicity of $\mathrm{G}$ yields
\begin{equation}
\label{strong_in_L1}
0 \le \limsup_{k\to\infty}\int_{0}^{\mathfrak{T}}\!\!\int\limits_{\Omega}
\big(\mathrm{G}(\theta_{nk},\mathbb{T}_{nk})-\mathrm{G}(\theta_{nk},\mathbb{T}_{n})\big):(\mathbb{T}_{nk}-\mathbb{T}_{n})\,\di  x\di t \le 0\,.
\end{equation}
Consequently, by Proposition \ref{weak_G(theta_k,T_k)},
\begin{equation}
\label{strong_in_L1_1}
\lim_{k\to\infty}\int_{0}^{\mathfrak{T}}\!\!\int\limits_{\Omega}\mathrm{G}(\theta_{nk},\mathbb{T}_{nk}):\mathbb{T}_{nk}\,\di x\di t
=\int_{0}^{\mathfrak{T}}\!\!\int\limits_{\Omega}\mathrm{G}(\theta_{n},\mathbb{T}_{n}):\mathbb{T}_{n}\,\di x \di t\,.
\end{equation}
\end{remark}
The following theorem plays a crucial role in the passage to the limit as $k\to\infty$.  Although it represents an analogous result to Proposition 4.14 in \cite{OwczarekWielgos23} (see also \cite{RoubicekL1}), for the sake of clarity of presentation, we decided to briefly include its proof here.
\begin{proposition}
\label{strong_T_nk}Let the assumptions of Theorem \ref{TWIERDZENIE} be satisfied. Then
\begin{align*}
		\mathbb{T}_{nk}\underset{k\to\infty}{\to} \mathbb{T}_{n}& \text{\quad in \quad}  L^{\infty}(0,\mathfrak{T};L^{2}(\Omega;\mathcal{S}^{3}))\,.
\end{align*}
\end{proposition}
\begin{proof}
Let $\zeta_{nk}\in L^{2}(0,\mathfrak{T};L^{2}(\Omega;\mathcal{S}^{3}))$ be an approximation of $\mathbb{T}_{n}$ with values in $\mathcal{V}^{k}$ such that
\begin{align}\label{new_pom_16}
		\zeta_{nk}	\underset{k\to\infty}{\to} \mathbb{T}_{n}
		& \text{\quad in \quad}  L^{2}(0,\mathfrak{T};L^{2}(\Omega;\mathcal{S}^{3}))\,.
\end{align}
Such an approximation exists due to the density of $\mathcal{V}^{k}$ in $L^{2}(0,\mathfrak{T};L^{2}(\Omega;\mathcal{S}^{3})$.

Choosing $\mathbf{v}=\mathbb{T}_{nk}-\zeta_{nk}$ in the equations \eqref{approx_nk_2} and \eqref{approx_n_2_0}, next subtracting them, we obtain the following energy identity for $\mathbb{T}_{nk}-\mathbb{T}_{n}$ (after the rearrangement of the terms)
\begin{align*}
		&\nonumber\int\limits_{\Omega} \mathbb{C}^{-1}\big(\mathbb{T}_{nk}(t)-\mathbb{T}_{n}(t)\big): \big(\mathbb{T}_{nk}(t)-\mathbb{T}_{n}(t)\big)\,\di x\\&\nonumber
		=
		-\int\limits_{0}^{t}\int\limits_{\Omega} \mathbb{C}^{-1}\big(\mathbb{T}_{nk,t}-\mathbb{T}_{n,t}\big): \big(\mathbb{T}_{n}-\zeta_{nk}\big)\,\di x\di t
		-\int\limits_{0}^{t}\int\limits_{\Omega}\textrm{G}(\theta_{nk},\mathbb{T}_{nk}):\mathbb{T}_{nk} \,\di x\di t\\&\nonumber\quad
		+\int\limits_{0}^{t}\int\limits_{\Omega}\textrm{G}(\theta_{nk},\mathbb{T}_{nk}):\zeta_{nk}\, \di x\di t
		+\int\limits_{0}^{t}\int\limits_{\Omega}\textrm{G}(\theta_{n},\mathbb{T}_{n}):\big(\mathbb{T}_{nk}-\zeta_{nk}\big)\, \di x\di t\\&\nonumber\quad
		+\int\limits_{0}^{t}\int\limits_{\Omega} \big(\varepsilon(u_{nk,t})-\varepsilon(u_{n,t})\big):\big(\mathbb{T}_{nk}-\mathbb{T}_{n}\big)\,\di x\di t\\&\nonumber\quad
		+\int\limits_{0}^{t}\int\limits_{\Omega} \big(\varepsilon(u_{nk,t})-\varepsilon(u_{n,t})\big):\big(\mathbb{T}_{n}-\zeta_{nk}\big)\,\di x\di t\\&\quad
		+	\int\limits_{\Omega} \mathbb{C}^{-1}\big(\mathbb{T}_{nk}(0)-\mathbb{T}_{n}(0)\big): \big(\mathbb{T}_{nk}(0)-\mathbb{T}_{n}(0)\big)\di x\,.
\end{align*}
Passing to the limit in each term: the contributions involving $\mathbb{T}_{nk,t}-\mathbb{T}_{n,t}$ vanish by weak convergence in $L^{2}([0,\mathfrak{T}]\times\Omega)$ and \eqref{new_pom_16}. Terms with $\textrm{G}(\cdot,\cdot)$ converge due to \eqref{strong_in_L1_1} and weak-strong convergence arguments. The parts with $\varepsilon(u_{nk,t})-\varepsilon(u_{n,t})$ vanish since $\displaystyle u_{nk,t} \underset{k\to\infty}{\to} u_{n,t}$ strongly in $L^{2}(0,\mathfrak{T};\mathcal{W}^{n})$. Eventually, the initial term equals zero, which completes the proof.
\end{proof}
Proposition \ref{strong_T_nk}, together with Lebesgue's dominated convergence theorem, yields the following result, which is the last proposition before passing to the limit with respect to $k$.
\begin{proposition}
Assume the assumptions of Theorem \ref{TWIERDZENIE}. Then the following convergence holds:
  \begin{align*}
    \mathrm{T}_{n}\big(\mathrm{G}(\theta_{nk},\mathbb{T}_{nk}):\mathbb{T}_{nk}\big)
	\underset{k\to\infty}{\to}
	\mathrm{T}_{n}\big(\mathrm{G}(\theta_{n},\mathbb{T}_{n}):\mathbb{T}_{n}\big) &\text{\quad in \quad} L^{2}(0,\mathfrak{T};L^{2}(\Omega;\mathbb{R}))\,.
\end{align*}
\end{proposition}
We are now in a position to pass to the limit as $k\to\infty$ with fixed $n\in\mathbb{N}$. Using the convergences established above, we obtain for a.e. $t\in[0,\mathfrak{T}]$ the following system of equations, satisfied 
\begin{align}\label{approx_n_1}
	&\int\limits_{\Omega} u_{n,tt}\varphi\,\di x 
	+\int\limits_{\Omega}\mathbb{T}_{n}:\varepsilon(\varphi)\,\di x 
	-\int\limits_{\Omega}\theta_{n}\diver\varphi\,\di x  
	= \int\limits_{\Omega} f_{n}\varphi\,\di x\,,
	\end{align}
	\begin{align}\label{approx_n_2}
	&\int\limits_{\Omega} \mathbb{C}^{-1}\mathbb{T}_{n,t}:\psi\,\di x
	+\int\limits_{\Omega}\textrm{G}(\theta_{n}, \mathbb{T}_{n}):\psi\,\di x
	=\int\limits_{\Omega}\varepsilon(u_{n,t}):\psi\,\di x\,,	
	\end{align}
	\begin{align}\label{approx_n_3}
	&\int\limits_{\Omega}\theta_{n,t}\phi\,\di x
	+\int\limits_{\Omega}\nabla\theta_{n} \nabla \phi\,\di x
	+\int\limits_{\Omega}\theta_{n}\diver{u_{n,t}}\phi \,\di x
	=\int\limits_{\Omega}\textrm{T}_{n}\big( \textrm{G}(\theta_{n}, \mathbb{T}_{n}):\mathbb{T}_{n}\big)\phi\,\di x
\end{align}
for $\varphi\in\mathcal{W}^{n}$, $\psi\in L^{2}(\Omega;\mathcal{S}^{3})$ and $\phi\in H^{1}(\Omega;\mathbb{R})$.

We observe that, thanks to the regularity properties established above and the standard regularity theory of weak solutions for parabolic equations, the equations \eqref{approx_n_2} and \eqref{approx_n_3} are satisfied for a.e. $(t,x)\in[0,\mathfrak{T}]\times\Omega$. Moreover, from Lemma \ref{positivity_of_theta}, we have $\theta_{n}(t,x)>0$ for a.e. $x\in\Omega$ and all $t\in[0,\mathfrak{T}]$. Finally, in view of Remark 3.7 from article \cite{OwczarekWielgos23}, the following additional observation holds.
\begin{remark}
    After a possible change of values on a set of measure zero, we obtain
    \begin{align*}
    u_{n}\in C([0,\mathfrak{T}];H_{0}^{1}(\Omega;\mathbb{R}^{3}))\,.
    \end{align*}
\end{remark}
\section{Convergence Analysis and Proof of the Main Theorem}
In this section, we consider the system at the level $n$. We first reformulate the energy equation as an entropy equation, which provides the necessary information to derive the total energy dissipation inequality. Building on this, we establish uniform estimates for the relevant sequences and justify the passage to the limit in the associated Bochner spaces. These steps allow us to complete the proof of the main theorem of the article.
    
\subsection{Rewriting the Energy Equation as an Entropy Equation in the Approximation}
We transform the energy equation \eqref{approx_n_3} in the approximate system into the entropy equation. Since, for the system depending only on $n$, we have $\theta_{n}(t,x)>0$ for almost every $x\in \Omega$ and all $t\in[0,\mathfrak{T}]$, we may divide the pointwise energy equation by $\theta_{n}$ to obtain the following corresponding entropy formulation
\begin{align} \label{entropy_for_n}
	\big(\ln\theta_{n}+\diver{u}_{n}\big)_{t}
	-\Delta\ln\theta_{n}
    = \frac{\textrm{T}_{n}\big(\textrm{G}(\theta_{n},\mathbb{T}_{n}):\mathbb{T}_{n}\big)}{\theta_{n}}
    +\abs{\nabla\ln\theta_{n}}^{2}\,.
\end{align}
We define $\tau_{n}=\ln\theta_{n}$. Multiplying equation \eqref{entropy_for_n} by a test function $\phi\in C_{0}^{\infty}([0,\mathfrak{T})\times\Omega)$ and integrating over yields
\begin{align}
  \label{entropy_for_n_1}
		-\int\limits_{0}^{\mathfrak{T}}\int\limits_{\Omega}(\tau_n+&\diver{u_n})\phi_{t}\,\di x\di t
		+\int\limits_{0}^{\mathfrak{T}}\int\limits_{\Omega}\nabla\tau_n \nabla \phi\,\di x\di t+\int\limits_{\Omega} (\tau_{0n}+\diver{u_{0n}})\phi(0,x)\,\di x\nn\\
        &
	=\int\limits_{0}^{\mathfrak{T}}\int\limits_{\Omega} e^{-\tau_n}\textrm{T}_{n}\big(\textrm{G}(e^{\tau_n},\mathbb{T}_n):\mathbb{T}_n\big)\phi\,\di x\di t
        + \int\limits_{0}^{\mathfrak{T}}\int\limits_{\Omega}\abs{\nabla\ln\theta_{n}}^{2}\phi\,\di x\di t\,.
\end{align}
In the subsequent subsections, we focus primarily on the entropy equation \eqref{entropy_for_n} rather than the energy equation \eqref{approx_n_3}. Our ultimate goal is to pass to the limit as $n\rightarrow\infty$ in equation \eqref{entropy_for_n_1}. To rigorously justify this limit passage, suitable bounds for the sequences appearing in \eqref{entropy_for_n_1} are required. These bounds will be derived in the following subsections.
\subsection{Estimates with respect to $n$}	

\begin{proposition}\label{Prop_estimates}
Assume the assumptions of Theorem \ref{TWIERDZENIE}. Then there exists a constant $C(\mathfrak{T})>0$, independent of $n$, such that the following uniform estimate holds:
\begin{align*}
		&\norm{u_{n,t}}_{L^{\infty}(0,\mathfrak{T};L^{2}(\Omega;\mathbb{R}^{3}))}+
        \norm{\mathbb{T}_{n}}_{L^{\infty}(0,\mathfrak{T};L^{2}(\Omega;\mathcal{S}^{3}))}+
        \norm{\theta_{n}}_{L^{\infty}(0,\mathfrak{T};L^{1}(\Omega;\mathbb{R}))}\nonumber\\[1ex]&
        +
        \norm{\tau_{n}}_{L^{\infty}(0,\mathfrak{T};L^{1}(\Omega;\mathbb{R}))}
        +
        \norm{\nabla\tau_{n}}_{L^{2}(0,\mathfrak{T};L^{2}(\Omega;\mathbb{R}^{3}))}\nonumber\\[1ex]&
        \quad+\norm{e^{-\tau_{n}}\mathrm{T}_{n}(\mathrm{G}(e^{\tau_{n}},\mathbb{T}_{n}):\mathbb{T}_{n})}_{L^{1}(0,\mathfrak{T};L^{1}(\Omega;\mathbb{R}))}
        \leq C(\mathfrak{T})\,.
\end{align*}
\end{proposition}
\begin{proof}
First, we derive the energy balance for the system \eqref{approx_n_1}-\eqref{approx_n_3} by choosing $\varphi=u_{n,t}$, $\psi=\mathbb{T}_{n}$ and $\phi=1$, and summing the resulting equations. Next, we integrate the entropy equation \eqref{entropy_for_n} over $\Omega$ and subtract it from the energy balance, which yields, for  $t\leq\mathfrak{T}$

\begin{align}
\label{estimate_n}
    &\frac{1}{2} \int\limits_{\Omega}\abs{u_{n,t}(t)}^{2}\di x
    +\frac{1}{2}\int\limits_{\Omega}\mathbb{C}^{-1}\mathbb{T}_{n}(t):\mathbb{T}_{n}(t)\,\di x
    +\int\limits_{\Omega} \big(\theta_{n}(t)-\tau_{n}(t)\big)\, \di x\nonumber\\&\quad
    +\int\limits_{0}^{t}\int\limits_{\Omega}\textrm{G}(\theta_{n},\mathbb{T}_{n}):\mathbb{T}_{n}\,\di x \di t
    +\int\limits_{0}^{t}\int\limits_{\Omega}e^{-\tau_{n}}\textrm{T}_{n}(\textrm{G}(e^{\tau_{n}},\mathbb{T}_{n}):\mathbb{T}_{n})\,\di x \di t
    +\int\limits_{0}^{t}\int\limits_{\Omega} \abs{\nabla\tau_{n}}^{2} \,\di x \di t\nonumber\\&
    =\frac{1}{2} \int\limits_{\Omega}\abs{u_{n,t}(0)}^{2}\di x
    +\frac{1}{2}\int\limits_{\Omega}\mathbb{C}^{-1}\mathbb{T}_{n}(0):\mathbb{T}_{n}(0)\,\di x
    +\int\limits_{\Omega} \big(\theta_{n}(0)-\tau_{n}(0)\big) \di x\nonumber\\&\quad
    +\int\limits_{0}^{t}\int\limits_{\Omega}\textrm{T}_{n}\big(\textrm{G}(\theta_{n},\mathbb{T}_{n}):\mathbb{T}_{n}\big)\,\di x \di t
    +\int\limits_{0}^{t}\int\limits_{\Omega}f_{n}u_{n,t} \di x \di t\nonumber\\&
    \leq
    \frac{1}{2} \int\limits_{\Omega}\abs{P_{\mathcal{W}^{n}}u_{1}}^{2}\,\di x
    +\frac{1}{2}\int\limits_{\Omega}\mathbb{C}^{-1}P_{\mathcal{V}^{n}}\mathbb{T}_{0}:P_{\mathcal{V}^{n}}\mathbb{T}_{0}\di x
    +\int\limits_{\Omega} \big(\theta_{0n}-\ln\theta_{0n}\big)\, \di x\nonumber\\&\quad
    +\int\limits_{0}^{t}\int\limits_{\Omega}\textrm{G}(\theta_{n},\mathbb{T}_{n}):\mathbb{T}_{n}\,\di x \di t
    +C(\varepsilon)\norm{f_{n}}_{L^{2}((0,\mathfrak{T})\times\Omega)}^{2}+\varepsilon\norm{u_{n,t}}_{L^{\infty}(0,\mathfrak{T};L^{2}(\Omega))}^{2}\nonumber\\&
    \leq 
    \frac{1}{2} \int\limits_{\Omega}\abs{u_{1}}^{2}\di x
    +\frac{1}{2}\int\limits_{\Omega}\mathbb{C}^{-1}\mathbb{T}_{0}:\mathbb{T}_{0}\,\di x
    +\int\limits_{\Omega} \big(\theta_{0}-\ln\theta_{0}\big)\, \di x\\&\quad
    +\int\limits_{0}^{t}\int\limits_{\Omega}\textrm{G}(\theta_{n},\mathbb{T}_{n}):\mathbb{T}_{n}\,\di x \di t
    +\tilde{C}(\varepsilon)\norm{f}_{L^{2}((0,\mathfrak{T})\times\Omega)}^{2}+\varepsilon\norm{u_{n,t}}_{L^{\infty}(0,\mathfrak{T};L^{2}(\Omega))}^{2}\nonumber
\end{align}
for $\varepsilon>0$. In \eqref{estimate_n}, we have used that the initial conditions for $u_{n}$ and $\mathbb{T}_{n}$ in the approximate system are the orthogonal projections onto $\mathcal{W}^{n}$ and $\mathcal{V}^{n}$, respectively, along with the convergences \eqref{init_cond_nk_1} and \eqref{init_cond_nk_4}. Moreover, we observe that the fourth and fifth terms on the left-hand side of \eqref{estimate_n} is nonnegative for almost every $(t,x)\in [0,\mathfrak{T}]\times\Omega$. Taking the essential supremum in \eqref{estimate_n}, we conclude that the sequences appearing in the statement of the proposition are bounded by a constant independent of $n$, which completes the proof.
\end{proof} 
\begin{remark}\label{Remark_G(theta_n,T_n)_boundedness}
From the boundedness of the sequence \(\{\mathbb{T}_{n}\}_{n=1}^{\infty}\) in 
\(L^{2}([0,\mathfrak{T}]\times\Omega)\) obtained in Proposition \ref{Prop_estimates}  and the property \eqref{bound_G}, 
it follows that there exists a constant $C(\mathfrak{T})>0$ such that
\[
\|\mathrm{G}(\theta_{n},\mathbb{T}_{n})\|_{L^{2}(0,\mathfrak{T};L^{2}(\Omega;\mathcal{S}^{3}))} \le C(\mathfrak{T})\,.
\]
\end{remark}    
\begin{proposition}\label{Prop_tau_nt_boundedness}
		Let us assume assumptions of Theorem \ref{TWIERDZENIE}. Then, there exists a constant $C(\mathfrak{T})>0$ independent of $n$ such that the following estimate is satisfied
        \begin{align*}
        \norm{\tau_{n,t}}_{L^{1}(0,\mathfrak{T};(W^{1,s}(\Omega;\mathbb{R})^{*})}\leq C(\mathfrak{T})\,,
        \end{align*}
        where $s>3$.
\end{proposition}
\begin{proof}
Let us write again a pointwise entropy equation 
\begin{align*}
	\tau_{n,t}=
    \Delta\tau_{n}  
    +\abs{\nabla\tau_{n}}^{2}
    -\diver{u}_{n,t}
	+e^{-\tau_{n}}\textrm{T}_{n}\big(\textrm{G}(e^{\tau_{n}},\mathbb{T}_{n}):\mathbb{T}_{n}\big)\,.
\end{align*}
From Proposition \ref{Prop_estimates}, we know that $\nabla \tau_{n}$ is uniformly bounded in $L^{2}(0,\mathfrak{T};L^{2}(\Omega;\mathbb{R}^{3}))$, which, when interpreted in the weak sense, implies that $\Delta \tau_{n}$ is uniformly bounded in $L^{2}(0,\mathfrak{T};(W^{1,2}(\Omega;\mathbb{R}))^{*})$. In the same manner, the weak interpretation of the divergence shows that $\diver u_{n,t}$ is bounded in $L^{2}(0,\mathfrak{T};H^{-1}(\Omega;\mathbb{R})^{})$ by a constant independent of $n$. Furthermore, Proposition \ref{Prop_estimates} also yields the uniform boundedness of $\abs{\nabla\tau_{n}}^{2}$ in $L^{1}((0,\mathfrak{T})\times\Omega)$ as well as of  $\displaystyle e^{-\tau_{n}}\textrm{T}_{n}\big(\textrm{G}(e^{\tau_{n}},\mathbb{T}_{n}):\mathbb{T}_{n}\big)$ in $L^{1}((0,\mathfrak{T})\times\Omega)$. Using the Sobolev embedding theorem,
$$W^{1,s}(\Omega;\mathbb{R})\hookrightarrow C^{0,\lambda}(\Omega;\mathbb{R})\hookrightarrow C^{0}(\Omega;\mathbb{R})\hookrightarrow L^{\infty}(\Omega;\mathbb{R})\hookrightarrow (L^{1}(\Omega;\mathbb{R}))^{*}$$
for $s>3$ and $0\leq\lambda\leq 1-\frac{3}{s}$, hence
$$L^{1}(\Omega;\mathbb{R})\hookrightarrow (L^{1}(\Omega;\mathbb{R}))^{**}\hookrightarrow (W^{1,s}(\Omega;\mathbb{R}))^{*}\,.$$
In consequence, the sequence $\tau_{n,t}$ is bounded in $L^{1}(0,\mathfrak{T};(W^{1,s}(\Omega;\mathbb{R}))^{*})$ for $s>3$, with bounds independent of $n$. 
\end{proof}
\subsection{Convergence Properties of the $n$-th Approximation}
The uniform estimates with respect to $n\in\mathbb{N}$ from Propositions \ref{Prop_estimates} allow us to extract suitable subsequences in the relevant Bochner spaces. In particular, the following weak convergences hold:
\begin{align}
\label{convergenc_with_n_1}
	u_{n,t}\overset{*}{\underset{n\to\infty}{\rightharpoonup}} u_{t}& \text{\quad in \quad}  L^{\infty}(0,\mathfrak{T};L^{2}(\Omega;\mathbb{R}^{3}))\,,\\
    \label{convergenc_with_n_2}
	\mathbb{T}_{n}\overset{*}{\underset{n\to\infty}{\rightharpoonup}} \mathbb{T}& \text{\quad in \quad} L^{\infty}(0,\mathfrak{T};L^{2}(\Omega;\mathcal{S}^{3}))\,,\\
    \label{convergenc_with_n_3}
	\theta_{n}\overset{*}{\underset{n\to\infty}{\rightharpoonup}}\theta& \text{\quad in \quad} L^{\infty}(0,\mathfrak{T};\mathcal{M}^{+}(\overline{\Omega};\mathbb{R}))\,.
\end{align}
Moreover, to justify the passage to the limit in the system \eqref{approx_n_1}-\eqref{approx_n_3} together with \eqref{entropy_for_n_1}, additional convergence results for the nonlinear terms are required.
\begin{remark}\label{Remark_tau_n_convergence}
        From Proposition \ref{Prop_estimates}, combined with the Gagliardo-Nirenberg interpolation inequality, we obtain
        \begin{align*}
            \norm{\tau_{n}}_{L^{2}(\Omega)}\leq C_{GN}\norm{\nabla\tau_{n}}_{L^{2}(\Omega)}^{\frac{3}{5}}
            \norm{\tau_{n}}_{L^{1}(\Omega)}^{\frac{2}{5}}\,.
        \end{align*}
Consequently, by H\"older's inequality
\begin{align*}
            \norm{\tau_{n}}_{L^{2}([0,\mathfrak{T}]\times\Omega)}^{2}
            &\leq 
            C_{GN}^{2}\int\limits_{0}^{\mathfrak{T}}\norm{\nabla\tau_{n}}_{L^{2}(\Omega)}^{\frac{6}{5}}
            \norm{\tau_{n}}_{L^{1}(\Omega)}^{\frac{4}{5}}\di t \nonumber\\&
            \leq C_{GN}^{2} \mathfrak{T}^{\frac{2}{5}}\norm{\tau_{n}}_{L^{\infty}(0,\mathfrak{T};L^{1}(\Omega))}^{\frac{4}{5}}\norm{\nabla\tau_{n}}_{L^{2}(0,\mathfrak{T};L^{2}(\Omega))}^{\frac{6}{5}}\,.
\end{align*}
Hence, $\tau_{n}$ is uniformly bounded in $L^{2}(0,\mathfrak{T};H^{1}(\Omega;\mathbb{R}))$. In particular, there exists 
$\tau \in L^{2}(0,\mathfrak{T};H^{1}(\Omega;\mathbb{R}))$ such that, up to a subsequence,
\begin{align}
\label{weak_tau_n}
		\tau_{n}\underset{n\to\infty}{\rightharpoonup} \tau & \text{ \quad in \quad} L^{2}(0,\mathfrak{T};H^{1}(\Omega;\mathbb{R}))\,.
	\end{align}
\end{remark}
From Proposition \ref{Prop_tau_nt_boundedness}, Remark \ref{Remark_tau_n_convergence} and the Aubin-Lions theorem we have 
$$
    \{ \tau_{n}\in L^{2}(0,\mathfrak{T};H^{1}(\Omega;\mathbb{R})), \tau_{n,t} \in L^{1}(0,\mathfrak{T};(W^{1,s}(\Omega;\mathbb{R}))^{*}), s>3 \} \hookrightarrow\hookrightarrow L^{2}(0,\mathfrak{T};L^{2}(\Omega;\mathbb{R}))\,.
$$
 Consequently, strong convergence, at least for a subsequence, holds:
\begin{align}
\label{strong_tau_n}
        \tau_{n}\underset{n\to\infty}{\to} \tau & \text{ \quad in \quad} L^{2}(0,\mathfrak{T};L^{2}(\Omega;\mathbb{R}))\,.
\end{align}
Moreover, for a.e. $(t,x)\in [0,\mathfrak{T}]\times\Omega$, at least for a subsequence, pointwise convergence holds:
\begin{align*}
       \tau_{n}(t,x)\underset{n\to\infty}{\to} \tau(t,x)\,,
\end{align*}
    which implies
\begin{align}\label{Conv_theta_n_pointwise}
        \theta_{n}(t,x)=e^{\tau_{n}}(t,x)\underset{n\to\infty}{\to} e^{\tau}(t,x)\,.
\end{align}
By the argument in \cite{CieslakMuha_3D}, which relies on Egorov's theorem, the measure $\di \theta$ admits the following decomposition. Since $\displaystyle\theta_n \underset{n\to\infty}{\to} e^\tau$ almost everywhere, Egorov's theorem guarantees that for every $\varepsilon>0$ there exists a set $A_\varepsilon \subset (0,\mathfrak{T}) \times \Omega$ such that $|((0,\mathfrak{T}) \times \Omega) \setminus A_\varepsilon| < \varepsilon$ and $\displaystyle\theta_n \underset{n\to\infty}{\to} e^\tau$ uniformly on $A_\varepsilon$. This allows us to identify the continuous part of the measure:
    \begin{align*}
        \di \theta = e^\tau \di x\di t \quad \text{on \quad} A_\varepsilon\,.
    \end{align*}
By letting $\varepsilon \to 0$, we conclude that the continuous part satisfies
    \begin{align*}
        \di \theta = e^\tau \di x\di t \quad \text{a.e. in \quad} (0,\mathfrak{T}) \times \Omega\,.
    \end{align*}
However, Egorov's theorem only provides uniform convergence on sets of positive Lebesgue measure, so this argument does not rule out the presence of a singular part $g \ge 0$ of the measure $\theta$, which may be supported on sets of measure zero. Consequently by boundedness of $\theta_n$ in $L^{\infty}(0,\mathfrak{T};L^{1}(\Omega;\mathbb{R}))$, we have the decomposition
    \begin{align*}
        \di \theta(t) = e^\tau(t) \di x + g(t) \quad \text{a.e. in \quad} (0,\mathfrak{T}) \,,
    \end{align*}
where $g(t)$ is the singular part in $\Omega$ for almost all $t\in (0,\mathfrak{T})$, and in general, we cannot conclude that $\di \theta (t) = e^\tau(t) \di x$.

From Proposition \ref{Prop_estimates} and Remark \ref{Remark_G(theta_n,T_n)_boundedness}, we obtain the following results.
\begin{proposition}
\label{Remark_G_convergence}
Let the assumptions of Theorem \ref{TWIERDZENIE} be satisfied. Then the following weak convergence holds:
 \begin{align*}
	\textrm{G}(\theta_{n},\mathbb{T}_{n})
	\underset{n\to\infty}{\rightharpoonup}
	\gamma &\text{\quad in \quad} L^{2}(0,\mathfrak{T};L^{2}(\Omega;\mathcal{S}^{3}))\,,
	\end{align*}
where
 \begin{align*}
            \gamma=\int\limits_{\mathbb{R_+}\times\mathbb{R}^{6}}\textrm{G}(\lambda_{1},\lambda_{2})\,\di \nu_{(t,x)}(\lambda_{1},\lambda_{2})
            =\int\limits_{\mathbb{R}^{6}}\textrm{G}(e^{\tau},\lambda_{2})\,\di \mu_{(t,x)}(\lambda_{2})\,.
\end{align*}
Here, $\nu_{(t,x)}:(0,\mathfrak{T})\times\Omega\to\mathbb{R_+}\times\mathbb{R}^{6}$ is a Young measure generated by the sequence $\{(\theta_{n},\mathbb{T}_{n})\}_{n=1}^{\infty}$, and $\mu_{(t,x)}:(0,\mathfrak{T})\times\Omega\to\mathbb{R}^{6}$ is a Young measure generated by the sequence $\{\mathbb{T}_{n}\}_{n=1}^{\infty}$, with
\begin{align*}
    \nu_{(t,x)} = \delta_{e^{\tau}(t,x)} \otimes \mu_{(t,x)}\,.
\end{align*}
\end{proposition}
\begin{proof}
From Remark \ref{Remark_G(theta_n,T_n)_boundedness}, the sequence $\{\textrm{G}(\theta_{n},\mathbb{T}_{n})\}_{n=1}^{\infty}$ is uniformly bounded in \\ $L^{2}(0,\mathfrak{T};L^{2}(\Omega;\mathcal{S}^{3}))$, and hence, up to a subsequence, 
  \begin{align}\label{G(theta_n,T_n)_conv_gamma}
    \textrm{G}(\theta_{n},\mathbb{T}_{n})
	\underset{n\to\infty}{\rightharpoonup}
	\gamma &\text{\quad in \quad} L^{2}(0,\mathfrak{T};L^{2}(\Omega;\mathcal{S}^{3}))\,.
    \end{align}
To identify the limit, we employ Young measures. From Proposition \ref{Prop_estimates} we know that both of the sequences $\{\theta_{n}\}_{n=1}^{\infty}$ and $\{\mathbb{T}_{n}\}_{n=1}^{\infty}$ are bounded, which ensures that $\{ (\theta_{n},\mathbb{T}_{n})\}_{n=1}^{\infty}$ generates a Young measure
$$
\nu_{(t,x)}:(0,\mathfrak{T})\times\Omega \to \mathcal{M}(\mathbb{R_+}\times\mathbb{R}^{6})\,,
$$
while $\{\mathbb{T}_{n}\}_{n=1}^{\infty}$ generates a Young measure
\[
\mu_{(t,x)}:(0,\mathfrak{T})\times\Omega \to \mathcal{M}(\mathbb{R}^{6})\,.
\]
Moreover, $\theta_{n} \underset{n\to\infty}{\to} e^{\tau}$ pointwise a.e. in $(0,\mathfrak{T})\times\Omega$, so that (see \cite[Corollary 3.4]{MUller1999})
\[
\nu_{(t,x)} = \delta_{e^{\tau}(t,x)} \otimes \mu_{(t,x)} \quad \text{a.e. in \quad} (0,\mathfrak{T})\times\Omega\,.
\]
The function $\textrm{G}$ is continuous, i.e., $\textrm{G}\in C(\mathbb{R_+}\times\mathbb{R}^{6})$; hence, applying Theorem 3.1 in \cite{MUller1999} to a measurable set $A \subset (0,\mathfrak{T})\times\Omega$, it follows that
\begin{align}\label{G(theta_n,T_n)_conv_measure}
\textrm{G}(\theta_{n},\mathbb{T}_{n}) \underset{n\to\infty}{\rightharpoonup}  \overline{\textrm{G}} \quad \text{in} \quad L^{1}(A)\,,
\end{align}
where
\[
\overline{\textrm{G}}(t,x) = \langle \nu_{(t,x)}, \textrm{G} \rangle = \int_{\mathbb{R_+}\times\mathbb{R}^{6}} \textrm{G}(\lambda_{1},\lambda_{2}) \, \di \nu_{(t,x)}(\lambda_{1},\lambda_{2})\,.
\]
By uniqueness of the weak limit, from \eqref{G(theta_n,T_n)_conv_gamma} and \eqref{G(theta_n,T_n)_conv_measure}, we deduce
\begin{align*}
\gamma = \int_{\mathbb{R_+}\times\mathbb{R}^{6}} \textrm{G}(\lambda_{1},\lambda_{2}) \, \di \nu_{(t,x)}(\lambda_{1},\lambda_{2})
= \int_{\mathbb{R}^{6}} \textrm{G}(e^{\tau},\lambda_{2}) \, \di \mu_{(t,x)}(\lambda_{2})\,.
\end{align*}
This completes the proof.
\end{proof}
\subsection{Limit passage as $n\to\infty$}
In this subsection we carry out the limit passage as $n\to\infty$ in the approximate system. The procedure will be performed equation by equation, in order to carefully justify the convergence of all terms involved. We begin with the approximate force balance equation \eqref{approx_n_1}.

For any $N_{N}\in\mathbb{N}$ and $\varphi\in C_{0}^{\infty}\big([0,\mathfrak{T});C_{0}^{\infty}(\Omega;\mathbb{R}^{3}) \big)$ of the form
$\displaystyle\varphi(t)=\varphi_{1}(t)\sum\limits_{m=1}^{N_{N}}\varphi_{m}\,,$ with $\varphi_{1}\in C_{0}^{\infty}([0,\mathfrak{T}))$ and $\varphi_{m}\in\mathcal{W}^{n}$ we can write
\begin{align*}
        \int\limits_{0}^{\mathfrak{T}}\int\limits_{\Omega} u_{n,tt}\varphi\, \di x \di t
        = -\int\limits_{0}^{\mathfrak{T}}\int\limits_{\Omega} u_{n,t}\varphi_{t} \,\di x \di t
        -  \int\limits_{\Omega} u_{n,t}(0)\varphi(0) \di x\,.  
\end{align*}
 Using the convergences \eqref{convergenc_with_n_1}-\eqref{convergenc_with_n_3} together with the fact that $u_{n,t}(0)=P_{\mathcal{W}^{n}}u_{1}\underset{n\to\infty}{\to}u_{1}$ in $L^{2}(\Omega;\mathbb{R}^{3})$, we arrive at
        \begin{align*}
        \int\limits_{0}^{\mathfrak{T}}\int\limits_{\Omega} u_{n,tt}\varphi\, \di x \di t
        \underset{n\to\infty}{\to}
        -\int\limits_{0}^{\mathfrak{T}}\int\limits_{\Omega} u_{t}\varphi_{t} \,\di x \di t
        -  \int\limits_{\Omega} u_{1}\varphi(0) \,\di x\,,
    \end{align*}
    which holds in particular for all $\varphi\in C_{0}^{\infty}([0,\mathfrak{T})\times\Omega)$. Moreover, applying again the convergences \eqref{convergenc_with_n_1}-\eqref{convergenc_with_n_3} together with property \eqref{Property_f_n} we deduce
\begin{align*}      \int\limits_{0}^{\mathfrak{T}}\int\limits_{\Omega}\mathbb{T}_{n}:\varepsilon(\varphi)\,\di x \di t
        \underset{n\to\infty}{\to}
        \int\limits_{0}^{\mathfrak{T}}\int\limits_{\Omega}\mathbb{T}:\varepsilon(\varphi)\,\di x \di t\,,
    \end{align*}
        \begin{align*}
        \int\limits_{0}^{\mathfrak{T}}\int\limits_{\Omega}\theta_{n}:\diver{ }\varphi\,\di x \di t
        \underset{n\to\infty}{\to}
        \int\limits_{0}^{\mathfrak{T}}\langle\theta,\diver\varphi\rangle_{[\mathcal{M}^{+};C](\overline{\Omega})}\, \di t\,,
    \end{align*}
    and
\begin{align*}
        \int\limits_{0}^{\mathfrak{T}}\int\limits_{\Omega}f_{n}\varphi\,\di x \di t
        \underset{n\to\infty}{\to}
        \int\limits_{0}^{\mathfrak{T}}\int\limits_{\Omega}f\varphi\,\di x \di t
\end{align*}
for all test functions $\varphi\in C_{0}^{\infty}\big([0,\mathfrak{T});H_{0}^{1}(\Omega;\mathbb{R}^{3}) \big)$ of the form $\displaystyle\varphi(t)=\varphi_{1}(t)\sum\limits_{m=1}^{N_{N}}\varphi_{m}$, where $\varphi_{1}\in C_{0}^{\infty}([0,\mathfrak{T}))$ and $\varphi_{m}\in\mathcal{W}^{n}$ and consequently also for all $\varphi\in C_{0}^{\infty}([0,\mathfrak{T})\times\Omega)$.

Passing to the limit as $n\to\infty$ in \eqref{approx_n_1}, we thus obtain 
\begin{align*}
	   -\int\limits_{0}^{\mathfrak{T}}\int\limits_{\Omega} u_{t}\varphi_{t}\,\di x  \di t
	   &+\int\limits_{0}^{\mathfrak{T}}\int\limits_{\Omega}\mathbb{T}:\varepsilon(\varphi)\,\di x \di t
	   -\int\limits_{0}^{\mathfrak{T}}\langle\theta,\diver\varphi\rangle_{[\mathcal{M}^{+};C](\overline{\Omega})}\,\di t  \nonumber\\&\quad
	   = \int\limits_{0}^{\mathfrak{T}}\int\limits_{\Omega} f\varphi\,\di x\di t
       +\int\limits_{\Omega}u_{1}\varphi(0,x)\,\di x
\end{align*}
for every function $\varphi\in C_{0}^{\infty}([0,\mathfrak{T})\times\Omega)$, where $\theta\in L^{\infty}(0,\mathfrak{T};\mathcal{M}^{+}(\overline{\Omega}))$. 

Having already derived the limit form of the force balance equation \eqref{approx_n_1}, we now aim to pass to the limit in the remaining relations, namely \eqref{approx_n_2} and \eqref{entropy_for_n_1}. Before doing so, however, let us take a closer look at the convergences established in the previous subsection, as they will play a crucial role in the forthcoming arguments.

Due to the two-level Galerkin approximation, the approximate inelastic constitutive equation \eqref{approx_n_2} is satisfied pointwise almost everywhere, i.e.,
\begin{align}\label{approx_n_2_pointwise}
\mathbb{C}^{-1}\mathbb{T}_{n,t}(t,x)
	+\textrm{G}(\theta_{n}, \mathbb{T}_{n})(t,x)
	=\varepsilon(u_{n,t}(t,x))
\end{align}
for almost all $(t,x)\in (0,\mathfrak{T})\times\Omega$.	By integrating equality \eqref{approx_n_2_pointwise} over the time interval $(0,t)$, with $t\leq\mathfrak{T}$ we obtain 
\begin{align}\label{approx_n_2_pointwise_1}
\varepsilon(u_{n}(t,x))=\mathbb{C}^{-1}\mathbb{T}_{n}(t,x)+\varepsilon(u_{n}(0,x))-\mathbb{C}^{-1}\mathbb{T}_{n}(0,x)
	+\int\limits_0^t\textrm{G}(\theta_{n}, \mathbb{T}_{n})\,\di\tau\,.	
\end{align}
Since the sequence $\{\mathbb{C}^{-1}\mathbb{T}_{n}\}_{n=1}^\infty$ is bounded in  $L^{\infty}(0,\mathfrak{T};L^{2}(\Omega;\S^3))$ and the sequence\\ $\{\textrm{G}(\theta_{n},\mathbb{T}_{n})\}_{n=1}^\infty$ is bounded in $L^{2}(0,\mathfrak{T};L^{2}(\Omega;\mathcal{S}^{3}))$, we conclude that $\{\varepsilon(u_{n})\}_{n=1}^\infty$ is bounded in $L^{\infty}(0,\mathfrak{T};L^{2}(\Omega;\S^3))$, and hence $\{u_{n}\}_{n=1}^\infty$ is bounded in $L^{\infty}(0,\mathfrak{T};H^{1}_0(\Omega;\R^3))$. Therefore
$$\mathbb{C}^{-1}\mathbb{T}_{n}-\varepsilon(u_{n}) \underset{n\to\infty}{\rightharpoonup} \mathbb{C}^{-1}\mathbb{T}-\varepsilon(u) \quad\textrm{in}\quad L^{2}((0,\mathfrak{T})\times\Omega)\,.$$
Additionally, from \eqref{approx_n_2} and Remark \ref{Remark_G(theta_n,T_n)_boundedness}, we infer that 
$$\norm{\big(\mathbb{C}^{-1}\mathbb{T}_{n}-\varepsilon(u_{n})\big)_{t}}_{L^{2}((0,\mathfrak{T})\times\Omega)}\leq C(\mathfrak{T})$$
and hence
$$\big(\mathbb{C}^{-1}\mathbb{T}_{n}-\varepsilon(u_{n})\big)_{t}  \underset{n\to\infty}{\rightharpoonup} \gamma \quad\textrm{in}\quad L^{2}((0,\mathfrak{T})\times\Omega)\,,$$
where $\gamma$ is defined in Proposition \ref{Remark_G_convergence}. Indeed, from \eqref{approx_n_2_pointwise} we have
\begin{align}\label{after_approx_n_2}
	\int\limits_{0}^{\mathfrak{T}}\int\limits_{\Omega} \big(\mathbb{C}^{-1}\mathbb{T}-\varepsilon(u)\big)_{t}:\psi\,\di x \di t
	&=    
    -\int\limits_{0}^{\mathfrak{T}}\int\limits_{\Omega}\gamma:\psi\,\di x\di t 
\end{align}
for every function $\psi\in C_{0}^{\infty}([0,\mathfrak{T})\times\Omega)$.

It remains to perform the limit passage in the entropy equation, namely \eqref{entropy_for_n_1}. Based on the convergences \eqref{convergenc_with_n_1}-\eqref{convergenc_with_n_3}, it follows that there exists a measure $\sigma\in \mathcal{M}^{+}([0,\mathfrak{T}]\times\overline{\Omega})$ such that
\begin{align*}
		\abs{\nabla\tau_{n}}^{2}\overset{*}{\underset{n\to\infty}{\rightharpoonup}} \sigma & \text{ \quad in \quad} \mathcal{M}^{+}([0,\mathfrak{T}]\times\overline{\Omega})\,.
 \end{align*} 
Hence, in view of Proposition \ref{Prop_estimates}, we obtain that the sequence 
\begin{equation*}
 \big\{ e^{-\tau_{n}}\textrm{T}_{n}\big(\textrm{G}(e^{\tau_{n}},\mathbb{T}_{n}):\mathbb{T}_{n}\big)\big\}_{n=1}^\infty\qquad \mathrm{ is\, bounded\, in }\quad L^{1}((0,\mathfrak{T})\times\Omega)
\end{equation*} 
and there exists a measure $\tilde{\sigma}\in \mathcal{M}^{+}([0,\mathfrak{T}]\times\overline{\Omega})$ such that
\begin{align*}
		e^{-\tau_{n}}\textrm{T}_{n}\big(\textrm{G}(e^{\tau_{n}},\mathbb{T}_{n}):\mathbb{T}_{n}\big)\overset{*}{\underset{n\to\infty}{\rightharpoonup}} \tilde{\sigma} & \text{ \quad in \quad} \mathcal{M}^{+}([0,\mathfrak{T}]\times\overline{\Omega})\,.
 \end{align*} 
Ultimately, by taking the limit $n\rightarrow\infty$ in equation \eqref{entropy_for_n_1}, we arrive at
\begin{align}
  \label{entropy_weak_form}
		-\int\limits_{0}^{\mathfrak{T}}\int\limits_{\Omega}(\tau+&\diver{u})\phi_{t}\,\di x\di t
		+\int\limits_{0}^{\mathfrak{T}}\int\limits_{\Omega}\nabla\tau \nabla \phi\,\di x\di t+\int\limits_{\Omega} (\tau_{0}+\diver{u_{0}})\phi(0,x)\,\di x		
        \nonumber\\&
		=\langle \tilde{\sigma},\phi \rangle_{[\mathcal{M}^{+};C]([0,\mathfrak{T}]\times\overline{\Omega})}+\langle \sigma,\phi \rangle_{[\mathcal{M}^{+};C]([0,\mathfrak{T}]\times\overline{\Omega})}
\end{align}
for all  $\phi\in C_{0}^{\infty}([0,\mathfrak{T})\times\Omega)$. Hence, the limit passage in \eqref{entropy_for_n_1} is justified. Together with the results obtained for \eqref{approx_n_1} and \eqref{approx_n_2}, this completes the passage to the limit in all equations of the approximate system; thus, all limit relations required in the statement of Theorem \ref{TWIERDZENIE} have been established.
We additionally show that $\sigma\geq\abs{\nabla\tau}^{2}$. Following the explanation from \cite{CieslakMuha_3D}, for $\phi\in C([0,\mathfrak{T}]\times\overline{\Omega})$, we define a mapping 
\begin{align*}
I_{\phi}(f)=\int\limits_{(0,\mathfrak{T})\times\Omega}f^{2}\phi\,\di x\di t\,.
\end{align*}
This mapping is continuous, convex for $\phi\geq 0$, and lower semicontinuous in strong $L^{2}((0,\mathfrak{T})\times\Omega)$ topology. Using Mazur's lemma, we obtain that $I_{\phi}$ is weakly lower semicontinuous, and we can write 
\begin{align*}
\int\limits_{(0,\mathfrak{T})\times\Omega}\abs{\nabla\tau}^{2}\phi\,\di x\di t=I_{\phi}(\nabla\tau)\leq \liminf\limits_{n\to\infty}I_{\phi}(\nabla\tau_{n})=
\lim\limits_{n\to\infty}\int\limits_{(0,\mathfrak{T})\times\Omega}\abs{\nabla\tau_{n}}^{2}\phi\,\di x=
\langle \sigma,\phi \rangle\,,    
\end{align*}
which proves the statement.
\subsection{Total Energy Dissipation Inequality and Proof of the Main Result}
As a final part of the proof of Theorem \ref{TWIERDZENIE}, we show that the inequality of total energy dissipation \eqref{ineq_tot_en_diss} holds for the solution $\big( u,\mathbb{T},\tau \big)$. This section thus represents the last step in the proof of our main result.

Before doing so, let us first observe that, by testing $\varphi=u_{n,t}$ in \eqref{approx_n_1}, $\psi=\mathbb{T}_{n}$ in \eqref{approx_n_2} and $\phi=1$ in \eqref{approx_n_3}, and performing calculations analogous to those in Subsection 1.2, we arrive at the following total energy dissipation inequality for the approximate solutions  $\big(u_{n,t},\mathbb{T}_{n},\theta_{n}\big)$, valid for all 
$t\in[0,\mathfrak{T}]$:
\begin{align}\label{approx_n_ineq}
    &\int\limits_{\Omega}\big(\theta_{n}(t)-\tau_{n}(t)\big)\,\di x
    +\frac{1}{2} \int\limits_{\Omega}\abs{u_{n,t}(t)}^{2}\,\di x
    +\frac{1}{2}\int\limits_{\Omega}\mathbb{C}^{-1}\mathbb{T}_{n}(t):\mathbb{T}_{n}(t)\,\di x
    \nonumber\\&\quad
    +\int\limits_{0}^{t}\int\limits_{\Omega}\abs{\nabla\tau_{n}}^{2} \,\di x \di t    
    \leq\int\limits_{0}^{t}\int\limits_{\Omega}f_{n}u_{n,t}\,\di x \di t
    +\int\limits_{\Omega}\big(\theta_{0n}-\ln\theta_{0n}\big)\,\di x
     \\&\quad
    +\frac{1}{2} \int\limits_{\Omega}\abs{u_{1n}}^{2}\,\di x
    +\frac{1}{2}\int\limits_{\Omega}\mathbb{C}^{-1}\mathbb{T}_{0n}:\mathbb{T}_{0n}\,\di x\,,\nonumber
\end{align}
where $\tau_{n}=\ln \theta_{n}$.

We now pass to the limit inferior as $n\to\infty$ in \eqref{approx_n_ineq}. The weak convergences of both sequences $\{u_{n,t}\}_{n=1}^{\infty}$ and $\{\mathbb{T}_{n}\}_{n=1}^{\infty}$ in $L^{\infty}(0,\mathfrak{T};L^{2}(\Omega))$, established in \eqref{convergenc_with_n_1} and \eqref{convergenc_with_n_2}, imply that, for almost every $t\in[0,\mathfrak{T}]$, the corresponding terms in \eqref{approx_n_ineq} are weakly lower semicontinuous along subsequences. Consequently, we obtain
\begin{align*}    
    \liminf\limits_{n\to\infty} \frac{1}{2} \int\limits_{\Omega}\abs{u_{n,t}(t)}^{2}\,\di x\geq
     \frac{1}{2} \int\limits_{\Omega}\abs{u_{t}(t)}^{2}\,\di x\,,
    \end{align*}
    \begin{align*}
     \liminf\limits_{n\to\infty}\frac{1}{2}\int\limits_{\Omega}\mathbb{C}^{-1}\mathbb{T}_{n}(t):\mathbb{T}_{n}(t)\,\di x\geq
     \frac{1}{2}\int\limits_{\Omega}\mathbb{C}^{-1}\mathbb{T}(t):\mathbb{T}(t)\,\di x\,.
\end{align*}
Moreover, from \eqref{convergenc_with_n_3}, we also have a weak convergence of $\{\theta_{n}\}_{n=1}^{\infty}$ in $L^{\infty}(0,\mathfrak{T};\mathcal{M}^{+}(\overline{\Omega}))$ and from \eqref{Conv_theta_n_pointwise} a pointwise convergence for almost every $(t,x)\in(0,\mathfrak{T})\times\Omega$. Let us define $E$ as the set of all points where $\{\theta_{n}\}_{n=1}^{\infty}$ does not converge to $\theta$ and for each $t\in(0,\mathfrak{T})$, let $E_{t}=\{ x\in\overline{\Omega}: (t,x)\in E \}$. Clearly, $\abs{E}=0$ in $(0,\mathfrak{T})\times\Omega$, which implies that $\abs{E_{t}}=0$ in $\Omega$ for almost every $t\in(0,\mathfrak{T})$. 
 
 Therefore, for almost every $t\in (0,\mathfrak{T})$ and any non-negative $\phi\in C(\overline{\Omega})$, we have
\begin{align*} 
    \lim\limits_{n\to\infty}\int\limits_{\Omega}\theta_{n}(t)\phi\,\di x
    &=\lim\limits_{n\to\infty}\big( \int\limits_{\Omega\setminus E_{t}}\theta_{n}(t)\phi\,\di x + \int\limits_{E_{t}}\theta_{n}(t)\phi\,\di x \big)
    =\lim\limits_{n\to\infty}\int\limits_{\Omega\setminus E_{t}}\theta_{n}(t)\phi\,\di x
    \nonumber\\&=\langle \theta(t),\phi\rangle _{[\mathcal{M}^{+},C](\overline{\Omega})}
    =\int\limits_{\Omega}\phi\,\di \theta(t)\,,
\end{align*}
where the last equality follows from the decomposition of the measure\\ $\displaystyle\di \theta(t)=e^{\tau}(t)\di x +g(t)$ with $g(t)\geq0$ being the singular part supported on sets of Lebesgue measure zero.

In particular, taking $\phi=1$ as a test function yields
\begin{align*}    
    \int\limits_{\Omega}\theta_{n}(t)\,\di x
     \underset{n\to\infty}{\to}
    \int\limits_{\Omega}\di\theta(t)\,.
\end{align*}
Next, by \eqref{strong_tau_n}, the positive sequence $\{\tau_{n}\}_{n=1}^{\infty}$ converges strongly in $L^{2}((0,\mathfrak{T})\times\Omega)$. Consequently, for almost every $t\in[0,\mathfrak{T}]$, a subsequence converges in $L^{1}(\Omega)$, which yields
\begin{align*}
    \int\limits_{\Omega}\tau_{n}(t)\,\di x
     \underset{n\to\infty}{\to}
     \int\limits_{\Omega}\tau(t)\,\di x\,.
\end{align*}
From the weak semi-continuity of the norm in $L^{2}((0,\mathfrak{T})\times\Omega)$ and Remark \ref{Remark_tau_n_convergence} we get
\begin{align*}
\liminf\limits_{n\to\infty}\int\limits_{0}^{t}\int\limits_{\Omega}\abs{\nabla\tau_{n}}^{2} \di x \di t
        \geq 
        \int\limits_{0}^{t}\int\limits_{\Omega}\abs{\nabla\tau}^{2} \di x \di t   
\end{align*}
and from the convergences \eqref{Property_f_n} and \eqref{convergenc_with_n_1} it follows that
\begin{align*}
        \int\limits_{0}^{t}\int\limits_{\Omega}f_{n}u_{n,t}\,\di x \di t
        \underset{n\to\infty}{\to}
        \int\limits_{0}^{t}\int\limits_{\Omega}fu_{t}\,\di x \di t\,.
\end{align*}   
For the initial conditions, from \eqref{init_cond_nk_2} - \eqref{init_cond_nk_4}, we obtain
    \begin{align*}
        \frac{1}{2} \int\limits_{\Omega}\abs{u_{1n}}^{2}\di x
        \underset{n\to\infty}{\to}
        \frac{1}{2} \int\limits_{\Omega}\abs{u_{1}}^{2}\di x\,,
    \end{align*}
    \begin{align*}
       \frac{1}{2}\int\limits_{\Omega}\mathbb{C}^{-1}\mathbb{T}_{0n}:\mathbb{T}_{0n}\,\di x
        \underset{n\to\infty}{\to}
        \frac{1}{2}\int\limits_{\Omega}\mathbb{C}^{-1}\mathbb{T}_{0}:\mathbb{T}_{0}\,\di x
    \end{align*}
    and
\begin{align*}
       \int\limits_{\Omega}\theta_{0n}\,\di x
        \underset{n\to\infty}{\to}
        \int\limits_{\Omega}\theta_{0}\,\di x
        =\int\limits_{\Omega}e^{\tau_{0}}\,\di x\,.
\end{align*}   
Moreover, $\theta_{0n}\to \theta_0$ strongly in $L^1(\Omega;\R)$, hence
\begin{align*}
       \int\limits_{\Omega}\tau_{0n}\,\di x
        \underset{n\to\infty}{\to}
        \int\limits_{\Omega}\tau_{0}\,\di x\,.
\end{align*} 
Finally, we conclude that the total energy dissipation inequality \eqref{ineq_tot_en_diss} holds for the solution $\big( u,\mathbb{T},\tau \big)$. This establishes all the required properties, and therefore, the proof of Theorem \ref{TWIERDZENIE} is complete. With this, the main result of the paper is fully proven.\\[2ex]
  
{\bf Declaration of competing interest}\\[1ex]
The authors declare that they have no known competing financial interests or personal relationships that could have appeared to influence the work reported in this paper.\\[1ex]
  
\bibliographystyle{abbrv}
	\footnotesize{

	}
\end{document}